\documentclass[10pt]{article}  
\usepackage{amssymb}
\usepackage{stmaryrd}

\usepackage{amsmath,amsfonts,amssymb,amsthm, bm}

\newcommand{\pa}[1]{\left(#1\right)}
\newtheorem{theorem}{Theorem}[section]
\newtheorem{lemma}[theorem]{Lemma}
\newtheorem{proposition}[theorem]{Proposition}
\newtheorem{corollary}[theorem]{Corollary}
\newtheorem{remark}[theorem]{Remark}

\newtheorem{example}[theorem]{Example}
\renewcommand{\epsilon}{\varepsilon}

\usepackage{fullpage}
\usepackage{xcolor}

\begin{document}

\numberwithin{equation}{section}

\title{Local structure of gradient almost Ricci solitons with harmonic Weyl tensor}

\author{\quad V. Borges\footnote{Faculdade de Matemática, Universidade Federal do Pará, 66075-110, Belém, Pará, Brazil, valterborges@ufpa.br.  }   \quad M. A. R. M. Horácio\footnote{ Departamento de Matemática, Universidade de Brasília, 70910-900, Brasília-DF, Brazil, m.a.r.m.horacio@mat.unb.br.} \quad J. P. dos Santos\footnote{ Departamento de Matemática, Universidade de Brasília, 70910-900, Brasília-DF, Brazil, joaopsantos@unb.br.}}

\date{}

\maketitle{}

\begin{abstract}
In this article, we investigate a gradient almost Ricci soliton with harmonic Weyl tensor. We first prove that its Ricci tensor has at most three distinct eigenvalues of constant multiplicities in a neighborhood of a regular point of the potential function. Then, we classify those with exactly two distinct eigenvalues. It is worth mentioning that the case with exactly one eigenvalue has already been settled elsewhere. Our results are based on a local representation of these manifolds as multiply warped products of a one-dimensional base, having at most two Einstein fibers, which we also obtain in this paper. These results extend a result by Catino, who assumes, in addition, that the Weyl tensor is radially flat, and a result by Kim, who considers the four-dimensional case.

\end{abstract}
\vspace{0.2cm} 

 \section{Introduction and main results}
 
\hspace{.5cm}A gradient almost Ricci soliton $(M, g, f, \lambda)$ is a four-tuple composed of a Riemannian manifold $(M^{n},g)$, with $n\geq3$, and two smooth functions $f,\lambda \in C^{\infty}(M)$ satisfying
\begin{align}\label{fundeq}
	\mathrm{Ric} +  \nabla^2 f= \lambda g.
\end{align}

Gradient almost Ricci solitons were introduced in \cite{pigola}, and there are many works investigating this structure, see for example \cite{barros4,barros2,barros1,brasil,brgs1,cafega,ma}. If we consider $f$ constant, it is well known that $\lambda$ must be constant, and \eqref{fundeq} gives rise to the equation describing an Einstein manifold \cite{besse}. If, on the other hand, we consider $\lambda$ constant in equation \eqref{fundeq}, then $M$ must be a gradient Ricci soliton \cite{hamilton2}. It is common in the literature to consider Einstein manifolds as trivial Ricci solitons \cite{fega}. In any case, results in which either $f$ or $\lambda$ is proven to be constant will be referred to as {\it triviality} results. Equation \eqref{fundeq} also includes other interesting structures in addition to Ricci solitons. Our results apply to the case of gradient Schouten solitons, which are obtained by setting $\lambda=(1/2(n-1))R+\tau$ \cite{brgs3,catino}, where $R$ is the scalar curvature of $M$ and $\tau\in\mathbb{R}$.

It is well known that an Einstein manifold has constant sectional curvature if either $n=3$, or $n\geq4$, and its Weyl tensor vanishes. Recall that the last condition is equivalent to a manifold being locally conformally flat. Complete gradient Ricci solitons were classified under the assumption that either $n=3$ and $\lambda>0$, or $n\geq4$, local conformal flatness, and $\lambda\geq0$ (see \cite{cachen,caoqian,cccmm} and references therein). In the context of almost Ricci solitons, Catino obtained in \cite{catino0} a local representation as a warped product of a one-dimensional base and an Einstein fiber, assuming local conformal flatness.

The results mentioned above admit many different extensions by considering weaker conditions on the Weyl tensor (see, for instance, \cite{caoqian,cccmm,catino0,kim1,kim2,kim3,muse}). This work is concerned with the harmonicity of the Weyl tensor. In this case, there is no loss of generality in assuming $n\geq4$. The classification of complete gradient Ricci solitons with harmonic Weyl tensor and $\lambda>0$ was obtained by combining the results of \cite{fega} and \cite{muse}. When $\lambda=0$ and $n=4$, these solitons were classified in \cite{kim1}. Very recently, the same author obtained a classification in \cite{kim3}, when $\lambda\geq0$ and $n\geq5$. His technique also allowed him to provide a local representation as a multiply warped product, for any constant $\lambda\in\mathbb{R}$ and $n\geq4$. In the case of gradient almost Ricci solitons with harmonic Weyl tensor, Catino obtained in \cite{catino0} a local description of these metrics as a warped product, assuming in addition that $W(\nabla f,\cdot,\cdot,\cdot)\equiv0$. In dimension four, Kim \cite{kim2} obtained a full classification with no further assumptions.

The main goal of this article is to investigate gradient almost Ricci solitons assuming their Weyl tensors are harmonic and $n\geq4$. In our first result, stated below, we deal with the number of eigenvalues of the Ricci tensor. More precisely, we extend to all dimensions a result of Kim \cite{kim2}, where he showed that the number of distinct eigenvalues of the Ricci tensor on a four-dimensional gradient almost Ricci soliton is at most three.

\begin{theorem}\label{maxxnumeig-INTRO}
Let $(M^n,g,f,\lambda)$, $n\geq4$, be a gradient almost Ricci soliton with harmonic Weyl curvature with $f$ nonconstant. Then the Ricci tensor of $M$ has at most three distinct eigenvalues at each point of $M$. If $p\in M$ is a regular point of $f$, then there is an open neighborhood of $p$ such that the number of eigenvalues and their multiplicities are constant within this set.
\end{theorem}

As mentioned above, the result of Catino \cite{catino0} assuming harmonic Weyl tensor and $W(\nabla f,\cdot,\cdot,\cdot)\equiv0$ implies a local representation of $M$ as a warped product $I\times_{h}N^{n-1}$ with a one-dimensional base $I$ and an Einstein fiber $N^{n-1}$. In the next result, we describe the local geometry of a gradient almost Ricci soliton assuming it has harmonic Weyl tensor, but without imposing any other conditions. As we will see, in the general case we can have at most two fibers, and examples given by Kim in dimension four \cite{kim2} show that this estimate on the number of fibers cannot be improved. 

\begin{theorem}\label{decompwarpint}
Any nontrivial gradient almost Ricci soliton with harmonic Weyl curvature is locally a multiply warped product with a one-dimensional base and at most two fibers. Furthermore, the fibers of dimension at least two must be Einstein.
\end{theorem}

The proofs of Theorem \ref{maxxnumeig-INTRO} and Theorem \ref{decompwarpint} are based on showing that certain quantities defined in terms of the warping functions satisfy a nonconstant polynomial of degree at most two. This is quite different from the proofs presented in  \cite{kim1,kim2,kim3,feng}, and has the advantage of being more concise. See Proposition \ref{prop_n=4} and Lemma \ref{lem_ply_deg2_} below, for $n=4$ and $n\geq4$, respectively.

The precise definition of the number of fibers of a multiply warped product is provided in Remark \ref{grmnt_wrpngfnctn}, following the conventions presented in \cite{brovaz}. As illustrated in Example \ref{exwmf}, the sharp upper bound on this quantity established by Theorem \ref{decompwarpint} hinges critically on the harmonicity of the Weyl tensor. Indeed, Example \ref{exwmf} showcases families of gradient almost Ricci solitons on multiply warped product metrics that admit any prescribed number of fibers. 

When $(M^n,g,f,\lambda)$ has harmonic Weyl tensor, Theorem \ref{maxxnumeig-INTRO} ensures that its Ricci tensor has a constant number of distinct eigenvalues around a regular point of $f$, which is at most three. If this number is exactly $1$, then the metric is Einstein, and a classification was given in \cite[Theorem 2.3]{pigola}. Thus, it remains to classify the gradient almost Ricci solitons whose Ricci tensor has exactly either two or three eigenvalues. In dimension four, this was accomplished by Kim (see Theorems 1.2 and 1.3 of \cite{kim2}).

In our next result, we give a complete local classification when $n\geq4$ and the Ricci tensor has exactly two distinct eigenvalues.

\begin{theorem}\label{pendT14'}
	Let $(M^n, g, f, \lambda)$, $n \geq 4$, be a gradient almost Ricci soliton with $f$ nonconstant and harmonic Weyl curvature. Assume its Ricci tensor has exactly two eigenvalues. Then one of the following happens:
	\begin{enumerate}
		\item\label{item_1} $M^{n}$ is locally isometric to a warped product $I\times_{h}N^{n-1}$, with $N^{n-1}$ Einstein and the functions $f$ and $\lambda$ given by \eqref{eqitem_1_harm} and \eqref{eqitem_2_harm}, respectively.
        \item\label{item2'} $M^n$ is locally isometric to the product $B^{k} \times N^{n-k}$, where $(B^{k},g_{B})$ is Ricci-flat, $\nabla_{B} f$ is a vector field on $B$ satisfying $\nabla_{B}\nabla_{B} f=\lambda g_{B}$, $\lambda\neq0$ is a constant, and $N^{n-k}$ is Einstein with Ricci curvature $\lambda$.
	\end{enumerate}
\end{theorem}

Triviality results as those of item \ref{item2'} in Theorem \ref{pendT14'} for almost Ricci solitons, where $\lambda$ is proven to be constant, can be found for example in \cite[Corollary 2.6]{brgs1}, \cite[Theorem 1.4]{cafega} or \cite[Proposition 3.1]{ma}.

It is worth mentioning that complete manifolds $(B^{k},g_{B})$ carrying a homothetic nonparallel vector field $\nabla_{B} f$, i.e. satisfying $\nabla_{B}\nabla_{B} f=\lambda g_{B}$ with $\lambda\neq0$ constant, were completely classified in \cite{kerbrat}. Thus, assuming completeness in Item \ref{item2'} of Theorem \ref{pendT14'} gives rise to the following rigidity result, in the sense of \cite{petersen}.

\begin{corollary}\label{cor}
	Let $(M^n, g, f, \lambda)$, $n \geq 4$, be a complete gradient almost Ricci soliton with $f$ nonconstant, harmonic Weyl curvature and assume $W(\nabla f,\cdot,\cdot,\cdot)$ does not vanish identically. If its Ricci tensor has exactly two eigenvalues, then $\lambda\neq0$ is a constant, $M^n$ is isometric to a quotient of the rigid soliton $\mathbb{R}^{k}\times N^{n-k}$ with $k\geq2$, and $f(x,p)=\frac{1}{2}\lambda \vert x\vert^2+\left\langle v,x\right\rangle+c$, where $v\in\mathbb{R}^k$ and $c\in\mathbb{R}$.
\end{corollary}

Now, we consider the locally conformally flat case. It was proved by Catino in \cite[Corollary 1.4]{catino0} that they are locally isometric to warped products. As a consequence of our computations, we obtain expressions for $\lambda$ and $f$ in terms of the warping function. Our results, together with Catino’s, imply the following:

\begin{corollary}\label{pendT15}
	Let $(M^n, g, f, \lambda)$ be a gradient almost Ricci soliton. Then, it is locally conformally flat if and only if for any regular point $p\in M$ of $f$ there is a neighborhood of $p$ which is isometric to a warped product $I\times_{h}N^{n-1}$, where $h:I\rightarrow\mathbb{R}$ is a positive warping function, $I$ is an interval, $N^{n-1}$ has constant sectional curvature, and the functions $f$ and $\lambda$ are given by \eqref{eqitem_1_harm} and \eqref{eqitem_2_harm}, respectively.
\end{corollary}

Thus, any warped product $I\times_{h}N^{n-1}$, with $N^{n-1}$ a space form, can be made into a locally conformally flat almost Ricci soliton.

Now we present another consequence of our results, concerning gradient Schouten solitons. As mentioned earlier, they are special instances of gradient almost Ricci solitons, and arise as self-similar solutions of the Schouten flow \cite{catino0,catino2,catino}. In \cite{brgs4}, the first author showed that complete gradient locally conformally flat Schouten solitons are rigid, in the sense of \cite{petersen}. Below, we extend this result, as an application of Theorem \ref{pendT14'}.
\begin{theorem}\label{scht_weyHar}
    If a complete gradient Schouten soliton has harmonic Weyl tensor, then it is rigid.
\end{theorem}

In what follows, we describe how this article is organized. In Section \ref{prelim}, we present preliminary results concerning manifolds whose Cotton tensor is Codazzi, and provide a few formulas relating geometric quantities of multiply warped products with their counterparts on the base and fibers. 

In Section \ref{cotton_results}, we investigate the gradient almost Ricci solitons whose Cotton tensor vanishes. We first show that $R$, $\lambda$, and the eigenvalues of $\mathrm{Ric}$ depend only on the arc length parameter $s$ associated with an integral curve of $\nabla f/\vert\nabla f\vert$. This is done in Subsection \ref{subsec_dep_s}. In Subsection \ref{der_ivey_ex}, we show that these solitons are locally represented as multiply warped products $I\times_{h_1} N_1^{r_1}\times\cdots\times_{h_k} N_{k}^{r_k}$, having a one-dimensional base and $k$ fibers, and that $N_i^{r_{i}}$ is Einstein if $r_i\geq2$. In the same subsection, we show that $k\leq2$ by constructing a nonconstant polynomial of degree at most two, which has $h'_{i}/h_{i}$ as a root for each $i\in\{1,\ldots,k\}$. We use this estimate to obtain, as a consequence, that the number of distinct eigenvalues of the Ricci tensor is at most three. These results prove Theorem \ref{maxxnumeig-INTRO} and Theorem \ref{decompwarpint}.

In Section \ref{subs_mult_warp}, we investigate more closely the almost Ricci solitons that are multiply warped products with at most two fibers, and have harmonic Weyl tensor. The case with only one fiber is considered in Subsection \ref{subs_mult_warp_1F}, and the case with two fibers is treated in Subsection \ref{subs_mult_warp_2F}. These results are used in subsequent proofs.

In Section \ref{quantity_multip}, we present the proofs of Theorem \ref{pendT14'}, Corollary \ref{cor}, Corollary \ref{pendT15} and Theorem \ref{scht_weyHar}. We start the section by investigating the case where the Ricci tensor has exactly two distinct eigenvalues. Regarding the proof of Theorem \ref{pendT14'}, two cases may occur according to the number of fibers. The first one is when the local geometry is that of only one fiber, which gives item \ref{item_1}. This case is simple to solve. The second one is when the local geometry is that of two fibers. In this case, the argument is more subtle. Part of the proof is by contradiction, and is based on writing $\lambda$ and the warping functions in terms of $f$ and its derivatives (see Lemma \ref{f_determines_ever} below), which allows us to conclude that $f$ satisfies two distinct ordinary differential equations. This, in turn, forces $f$ to be constant, which contradicts our initial assumption that $f$ is nonconstant. The case $r_{1}=1$ is simpler, while the case $r_{1}>1$ consists of constructing a polynomial $P(x)$ of degree $12$ such that $P(f'(s))=0,\ \forall s$, which is used to prove that $f$ is constant. The construction of $P(x)$ is involved, and is presented in the Appendix (see Lemma \ref{lem:nontrivial_polynomial}).

\section{Preliminaries}\label{prelim}

\hspace{.5cm}Given vector fields $X,Y,Z,T\in\mathfrak{X}(M)$, we recall that
\begin{align}
	&\mathrm{Rm}(X,Y,Z,T)=\left\langle\nabla_{Y}\nabla_{X}Z-\nabla_{X}\nabla_{Y}Z+\nabla_{[X,Y]}Z,T\right\rangle \nonumber,\\
	&W(X,Y,Z,T)=\mathrm{Rm}(X,Y,Z,T)-\frac{1}{n-2}\left(\left(\mathrm{Ric}-\frac{R}{n}g\right)\varowedge g\right)(X,Y,Z,T)-\frac{R}{2n(n-1)}(g\varowedge g)(X,Y,Z,T) \nonumber,\\
	&C(X, Y, Z) = \left(\nabla_X \mathrm{Ric}\right)(Y, Z) -  \left(\nabla_Y \mathrm{Ric}\right)(X, Z) -\frac{1}{2(n-1)} \left\{\left(\nabla_X (Rg)\right)(Y, Z) - \left(\nabla_Y (Rg)\right)(X, Z)\right\},\nonumber
\end{align}
are the curvature, Weyl and Cotton tensors, respectively, where $\varowedge $ is the Kulkarni-Nomizu product, whose definition can be found in \cite[page 47]{besse}, for example.

Recall that when $n=3$, the Weyl tensor vanishes identically. If $n\geq4$, then the vanishing of $W$ is equivalent to the Riemannian manifold being locally conformally flat. On the other hand, $M^3$ is locally conformally flat if and only if the Cotton tensor is identically zero. In higher dimensions, the vanishing of $C$ is equivalent to $M$ having harmonic Weyl tensor.

\subsection{Results for when the Weyl tensor is harmonic}

\hspace{.5cm}In this subsection, we recall Lemma 2.1 and Lemma 2.2 of \cite{kim1}. In the version of Lemma 2.1 below, we emphasize the fact that it is an equivalence.

\begin{lemma}[Lemma 2.1 of \cite{kim1}]\label{bari}
	Suppose $(M^n, g, f, \lambda)$ is a gradient almost Ricci soliton. Then $M$ has zero Cotton tensor if, and only if,
	\begin{equation}\tag{\ref{bari}}\label{harmWeylAlmost}
			\begin{aligned}
		\mathrm{Rm
}(\nabla f,X,Y,Z)&=Y\left(\frac{R}{2(n-1)}-\lambda\right)g(X,Z)-Z\left(\frac{R}{2(n-1)}-\lambda\right)g(X,Y)\nonumber\\
		&=\frac{1}{n-1}(\mathrm{Ric}(\nabla f,Y)g(X,Z)-\mathrm{Ric}(\nabla f,Z)g(X,Y))
	\end{aligned}
	\end{equation}
	for all vector fields $X,Y,Z \in \mathfrak{X}(M)$. If, in particular, $n\geq4$, the condition above is equivalent to the almost soliton having harmonic Weyl curvature.
\end{lemma}

\begin{lemma}[Lemma 2.2 of \cite{kim1}]\label{caochen}
	Let $(M,g,f,\lambda)$ be a gradient almost Ricci soliton with zero Cotton tensor and nonconstant $f$. Let $c$ be a regular value of $f$ and $\Sigma_{c}=f^{-1}(c)$ be the level surface of $f$. Then,
	\begin{enumerate}
		\item Where $\nabla f\neq0$, $E_{1}=\frac{\nabla f}{|\nabla f|}$ is an eigenvector of $\mathrm{Ric}$.
		\item\label{cao2} $|\nabla f|$ is constant on a connected component of $\Sigma_{c}$.
		\item\label{cao3} There is a function $s$ locally defined with $s(x)=\int\frac{\mathrm{d} f}{|\nabla f|}$, such that $\mathrm{d} s=\frac{\mathrm{d} f}{|\nabla f|}$ and $E_{1}=\nabla s$.
		\item\label{cao4} $E_{1}E_{1}f=-\mathrm{Ric}(E_{1},E_{1})+\lambda$. In particular, $-\mathrm{Ric}(E_{1},E_{1})+\lambda$ is constant on a connected component of $\Sigma_{c}$.
		\item\label{cao5} Near a point in $\Sigma_{c}$, the metric $g$ can be written as
		\begin{align*}
			g=\mathrm{d} s^2+\sum_{i,j\geq2}g_{ij}(s,x_{2},\ldots,x_{n}) \ \mathrm{d} x_{i}\otimes \mathrm{d} x_{j}.
		\end{align*}
		\item $\nabla_{E_{1}}E_{1}=0$.
	\end{enumerate}
\end{lemma}

Since a manifold has harmonic Weyl curvature precisely when its Schouten tensor $A=\mathrm{Ric}-\frac{R}{2(n-1)}g$ is Codazzi, we can apply Derdzi\'{n}ski's \cite{derd} results. On the open dense set
\begin{align}
	M_{\mathcal{A}}=\{x\in M \ |  \ E_{\mathcal{A}}(x)\text{ is constant in a neighborhood of }x\};\nonumber
\end{align}
the eigenvalues of $A$ are well defined and the corresponding eigenfunctions are smooth. We write them as $\lambda_1,\dots,\lambda_n$ and use eigenframes adapted to $\mathrm{Ric}$ in what follows. The next lemma records the identities we shall need.

\begin{lemma}[Derdzi\'{n}ski]\label{derdlemma}
	Let $(M^n,g)$, $n\geq4$, be a Riemannian metric with harmonic Weyl curvature. Let $\{E_{i}\}_{i=1}^{n}$ be a local orthonormal frame such that $\mathrm{Ric}(E_{i},\cdot)=\lambda_{i}g(E_{i},\cdot)$. Then,
	\begin{enumerate}
		\item\label{derd1} For any $i,j,k\geq1$,
		\begin{align*}
			(\lambda_{j}-\lambda_{k})\left\langle\nabla_{E_{i}}E_{j},E_{k}\right\rangle+\nabla_{E_{i}}(\mathcal{A}(E_{j},E_{k}))=
			(\lambda_{i}-\lambda_{k})\left\langle\nabla_{E_{j}}E_{i},E_{k}\right\rangle+\nabla_{E_{j}}(\mathcal{A}(E_{k},E_{i})).
		\end{align*}
		\item If $k\neq i$ and $k\neq j$, then $(\lambda_{j}-\lambda_{k})\left\langle\nabla_{E_{i}}E_{j},E_{k}\right\rangle=(\lambda_{i}-\lambda_{k})\left\langle\nabla_{E_{j}}E_{i},E_{k}\right\rangle$.
		\item Given distinct eigenfunctions $\lambda$ and $\mu$ of $\mathcal{A}$ and local vector fields $U$ and $V$ such that $A(V)=\lambda V$ and $A(U)=\mu U$ with $|U|=1$, it holds that $V(\mu)=(\mu-\lambda)\left\langle\nabla_{U}U,V\right\rangle$.
		\item Each distribution $D_{\lambda_{i}}$, defined by $D_{\lambda_{i}}(p)=\{v\in T_{p} M \ | \ \mathrm{Ric}(v,\cdot)=\lambda_{i}g(v,\cdot)\}$, is integrable and its leaves are totally umbilical submanifolds of $M$.
	\end{enumerate}
\end{lemma}

\subsection{Multiply warped products}

Since we will repeatedly use multiply warped product geometries, we recall their definition here. We write
\begin{equation}
	M = B \times_{h_1} N_1^{r_1} \times \cdots \times_{h_k} N_k^{r_k}, \qquad
	g_M = g_B + h_1^2 g_{N_1} + \cdots + h_k^2 g_{N_k};\nonumber
\end{equation}
and call $B$ the base, $N_i^{r_i}$ the fibers, and $h_i$ the warping functions. The particular case $k=1$ is the classical warped product; a standard reference on this is \cite{oneill}.

\begin{remark}[Counting fibers]\label{grmnt_wrpngfnctn}
	To count the fibers in a multiply warped product, we adopt the same conventions used in our earlier work on Ricci solitons with harmonic Weyl curvature \cite{brgs2025}. Specifically, a fiber with a constant warping function is absorbed into the base, and two factors whose warping functions differ by a positive constant multiple are identified after rescaling the corresponding fiber metric. Consequently, each $h_i$ is nonconstant, and the family $\{h_i\}_{1 \leq i \leq k}$ consists of pairwise nonproportional functions. The \textbf{number of fibers} is defined as the number of equivalence classes of warping functions under the equivalence relation $f \sim c\,g$ for $c > 0$.
\end{remark}

Metrics of this type appear throughout the literature \cite{bishop,brovaz,choui,danwan,derd2} and in many natural constructions \cite{caoqian,cccmm,caozhou,catino0,kerbrat,kim1,kim2,kim3,feng,shin}.

For later use, we shall now record the Levi-Civita connection, Ricci tensor, and scalar curvature of $M$ in terms of the geometry of $B$, the geometries of the fibers $N_i$, and the warping functions $h_i$. The following lemma compiles the formulas we will need.

\begin{lemma}[\cite{DobUn}]\label{warped}
	Let $M = B\times_{h_1} N_1^{r_1} \times \cdots \times_{h_k} N_{k}^{r_k}$ be a multiply warped product, and consider $X, Y, Z \in \mathcal{L}(B)$, $V \in \mathcal{L}(N_i)$ and $W \in \mathcal{L}(N_j)$ lifted vector fields. Then:
	
	\begin{enumerate}
		\item The covariant derivative of $M$ satisfies the following relations
		\begin{align*}
			&\nabla_X Y =\nabla^{B}_X Y ,\\
			&\nabla_X V = \nabla_V X = \frac{X(h_i)}{h_i} V,\\
			&\nabla_W V = 0,\ \ \text{if}\ \ i\neq j,\\
			&\nabla_W V = \nabla^{N_{i}}_W V-h_{i}g_{N_{i}}(W,V)\nabla^{B}h_{i},\ \ \text{if}\ \ i=j.
		\end{align*}
		
		\item The Ricci tensor of $M$ is given by
		\begin{align*}
			&\mathrm{Ric}(X, Y) = \mathrm{Ric}_{B}(X,Y) - \sum_{1 \leq \ell \leq k} \frac{r_{\ell}}{h_{\ell}} \nabla^{2}_{B}h_{\ell}(X, Y),\\
			&\mathrm{Ric}(X, V) = 0,\\
			&\mathrm{Ric}(V, W) = 0, \text{ if } i \neq j,\\
                &\mathrm{Ric}(V, W) = \mathrm{Ric}_{N_i}(V, W) - \left(\frac{\Delta_B h_i}{h_i} + (r_i - 1) \frac{| \nabla_{B} h_i |^2}{h_i^2} + \sum_{\substack{1 \leq \ell \leq k \\
					\ell \neq i  \\
			}} r_{\ell} \frac{g_{B}(\nabla_{B} h_i, \nabla_{B} h_{\ell})}{h_i h_{\ell}} \right)g(V, W), \text{if } i = j.
		\end{align*}
\item The scalar curvature of $M$ is given by
		\begin{equation*}
			\begin{aligned}
				R = R_{B} - 2 \sum_{1 \leq i \leq k} r_i \frac{\Delta_{B} h_i}{h_i} + \sum_{1 \leq i \leq k} \frac{R_{N_i}}{h_i^2} - \sum_{1 \leq i \leq k} r_i(r_i - 1) \frac{| \nabla_{B} h_i |^2}{h_i^2}-\sum_{1 \leq i \leq k} \sum_{\substack{1 \leq \ell \leq k \\
						\ell \neq i  \\
				}} r_i r_{\ell} \frac{g_{B}(\nabla_{B} h_i, \nabla_{B} h_{\ell})}{h_i h_{\ell}}.
			\end{aligned}
		\end{equation*}
	\end{enumerate}
\end{lemma}

\section{Gradient almost solitons whose Cotton tensor vanishes}\label{cotton_results}
\hspace{.5cm}The goal of this section is to obtain a local representation of $M$ as a multiply warped product. This is the content of Theorem \ref{decompwarp}. Next, we estimate the number of eigenvalues of the Ricci tensor of a gradient almost Ricci soliton with harmonic Weyl tensor. This is obtained in Corollary \ref{numbereigen}. As a consequence of this estimate, we conclude that at most two fibers can appear in this local representation, in the sense of Remark \ref{grmnt_wrpngfnctn}. These results prove both Theorem \ref{maxxnumeig-INTRO} and Theorem \ref{decompwarpint}.

\subsection{$R$, $\lambda$ and the eigenvalues of $\mathrm{Ric}$ depend only on $s$}\label{subsec_dep_s}

\hspace{.5cm}In this subsection, we show that using the system of coordinates of item \eqref{cao5} of Lemma \ref{caochen}, certain important functions depend only on $s$, the arc length of the integral curve of $\frac{\nabla f}{\vert\nabla f\vert}$.

First, we deal with the set of regular points of the potential function of the gradient almost Ricci soliton $M$. Suppose that it is nonconstant and consider the set $\mathcal{R}$ of all regular points of $f$,
\begin{align*}
	\mathcal{R}=\{x\in M \ | \ \nabla f(x)\neq0\}.
\end{align*}
Unlike in the gradient Ricci soliton case, where both $g$ and $f$ are real analytic in harmonic coordinates, a gradient almost Ricci soliton with harmonic Weyl curvature need not be analytic. Consequently, one cannot rely on analyticity to control the distribution of regular points of $f$. In what follows, we address this issue directly. Namely, we prove that on each connected component $W$ of the open and dense subset $M_{\mathcal A}$, the behavior of the regular set follows a clear alternative: either every point of $W$ is arbitrarily close to a regular point of $f$, or $f$ is constant in $W$. 
\begin{lemma}\label{denseempty}
	Let $(M^{n},g,f,\lambda)$ be a gradient almost Ricci soliton with harmonic Weyl curvature. If $f$ is not constant and $W$ is a connected component of $M_{\mathcal{A}}$, then either $W\cap\mathcal{R}$ is dense in $W$, or it is empty.
\end{lemma}
\begin{proof}
	Let $W$ be a connected component of $M_{\mathcal{A}}$ such that $W'=W\cap\mathcal{R}\neq\emptyset$. We will show that $W'$ is dense in $W$. Suppose by contradiction that this is not true and consider an open set $U\subset W\backslash W'$. Once $\nabla f$ vanishes in $U$, the almost Ricci soliton equation becomes $\mathrm{Ric}=\lambda g$ in $U$. As the number of eigenvalues of $\mathrm{Ric}$ is constant in $W$, it follows that $M$ is Einstein in this set, i.e., there is $\mu\in\mathbb{R}$ such that $\mathrm{Ric}=\mu g$ in $W$. In particular, $\nabla^2 f=(\lambda-\mu)g$ in $W$. Using the Bochner formula we obtain
	\begin{align*}
		X(\lambda)=\operatorname{div}(\nabla^2 f)(X)=\mathrm{Ric}(\nabla f,X)+X(\Delta f)=X(\mu f+n\lambda),\ \forall X\in\mathfrak{X}(W),
	\end{align*}
	which implies the existence of $c_{0}\in\mathbb{R}$ such that $\displaystyle\lambda=-\frac{\mu}{n-1}f+\mu+c_{0}$. Consequently,
	\begin{align*}
		\nabla^2 f=\left(-\frac{\mu}{n-1}f+c_{0}\right)g,\ \text{in}\ W.
	\end{align*}
	Now observe that if $\mu\neq0$, then $\displaystyle f-\frac{(n-1)c_{0}}{\mu}$ is an eigenfunction of $-\Delta$. If, on the other hand, $\mu=0$, then $\lambda$ is constant, $\mathrm{Ric}=0$ and $\nabla^2 f=\lambda g$, which means that $(W,g|_{W},f,\lambda)$ is a Ricci soliton. In both cases we conclude that $f$ is analytic in $W$. As $\nabla f$ vanishes in $U$, $f$ is constant in $W$, contradicting our assumption.
	
	This means that $W\backslash W'$ has empty interior or, equivalently, that the set $W'$ is dense in $W$.
\end{proof}

In the next lemma we prove only that $\lambda$ is constant on each connected component of a regular level $\Sigma_{c}$ of $f$, and hence $\lambda$ depends only on $s$ (cf. item \eqref{cao5} of Lemma \ref{caochen}). The corresponding statements for $R$ and for the first eigenvalue $\lambda_{1}$ have already been established: first for four-dimensional Ricci solitons in \cite{kim1}, and then for four-dimensional almost Ricci solitons in \cite{kim2}; one can easily check that the proof in \cite{kim2} applies in all dimensions.

\begin{lemma}\label{lemma3depend}
	Let $(M^n,g,f,\lambda)$, $n\geq4$, be a gradient almost Ricci soliton with harmonic Weyl curvature and nonconstant $f$. Then $R$, $\lambda$ and $\mathrm{Ric}(E_{1},E_{1})=\lambda_{1}$ are constant on a connected component of $\Sigma_{c}$, for a regular value $c$ of $f$.
\end{lemma}

\begin{proof}
	We already know that $\lambda-\lambda_{1}$ depends only on $s$, as established by item $(\ref{cao4})$ of Lemma \ref{caochen}, which ensures that $\lambda-\lambda_{1}=f''$. Since $\lambda_{1}$ also depends only on $s$ and we can write
	\begin{align*}
		\lambda&=(\lambda-\lambda_{1})+\lambda_{1},
	\end{align*}
	we see that $\lambda$ depends only on $s$ as well, and the lemma is proved.
\end{proof}

As in the case of gradient Ricci solitons, the eigenvalues of the Ricci tensor of a gradient almost Ricci soliton depend only on $s$. In dimension $n=4$, this result was obtained by Kim in \cite{kim2}. A proof can be obtained by simply adapting the arguments in the proof of \cite[Lemma 8]{brgs2025}, which are inspired by those in \cite[Lemma 3]{shin}. We state this fact in the following lemma, whose proof we omit.
\begin{lemma}
	Let $(M^n,g,f,\lambda)$, $n\geq4$, be a gradient almost Ricci soliton with harmonic Weyl curvature and nonconstant $f$. The functions $\lambda_{2},\ldots,\lambda_{n}$ are constant on each connected component of $\Sigma_{c}$. As a consequence, the functions $\xi_{2},\ldots,\xi_{n}$ are also constant on each connected component of $\Sigma_{c}$.
\end{lemma}

\subsection{Local multiply warped product structure}\label{der_ivey_ex}

\hspace{.5cm}In this subsection, we note that gradient almost Ricci solitons with zero Cotton tensor are multiply warped products around regular points of the potential function. We also establish that the number of fibers in the corresponding representation is at most two. \emph{Mutatis mutandis}, the same strategies we used in \cite{brgs2025} apply directly to our setting. Therefore, we have chosen to state the results without proofs.

\begin{theorem}\label{decompwarp}
Any gradient almost Ricci soliton with harmonic Weyl curvature is locally a multiply warped product $I\times_{h_1} N_1^{r_1}\times\cdots\times_{h_k} N_{k}^{r_k}$ of a one-dimensional base $I$ and $k$ fibers $N_{i}^{r_{i}}$. Furthermore, each fiber $N_{i}^{r_{i}}$ of dimension $r_{i}\geq2$ must be Einstein.
\end{theorem}

\begin{remark}
	As a straightforward consequence of the proof of Theorem \ref{decompwarp}, we get 
	\begin{equation}\label{eqXiRelFuncWarp}
		\frac{\lambda - \lambda_i}{f'} = \xi_i = \frac{h_{i}'}{h_i}
	\end{equation}
	This equation will be useful later on.
\end{remark}
We shall now estimate the number of fibers in the local representation given by Theorem \ref{decompwarp}. It will be convenient to introduce the functions
    \begin{align}\label{def_B_C}
        B=\frac{(n-1)\lambda-R+\lambda_{1}-(f')^2}{f'}\ \ \ \text{and}\ \ \ C=\frac{2(n-1)\lambda'-R'}{2(n-1)f'}+\lambda.
    \end{align}
Following the same steps as in \cite{brgs2025}, we obtain from the multiply warped decomposition that, for each $i\in\{1,\ldots,k\}$,
    \begin{align}\label{first_scnd_Dg}
        \xi^{2}_{i}-B\xi_i-C=-(r_i-1) \frac{\mu_i}{h_i^2}.
    \end{align}
This equation allows us to give an alternative proof of \cite[Proposition 3.4]{kim1} and of \cite[Proposition 3.4]{kim2}, which establish an estimate for the number of distinct eigenvalues of the Ricci tensor of a gradient almost Ricci soliton $M^4$ with harmonic Weyl tensor.

\begin{proposition}\label{prop_n=4}
    If $n=4$, then $M$ has at most two fibers in the representation given by Theorem \ref{decompwarp}. Equivalently, the Ricci tensor of $M^4$ has at most three distinct eigenvalues.
\end{proposition}

In what follows, we extend this argument to higher dimensions. Namely, we construct a nonzero polynomial of degree at most two, which has as roots the functions $\xi_{1},\ldots,\xi_{k}$.

\begin{lemma}\label{lem_ply_deg2_}
    The functions $\xi_{1},\ldots,\xi_{k}$ satisfy the following equation
    \begin{align}\label{poly_atmst2}
        B\xi_{i}^2+(B'+2\lambda)\xi_{i}+(C-\lambda)B+C'=0.
    \end{align}
    Furthermore, the polynomial in $\xi_{i}$ defined by the left-hand side of \eqref{poly_atmst2} is nontrivial.
\end{lemma}

As an immediate consequence, we obtain the following result. 

\begin{corollary}\label{numbfibalmo}
	Consider a multiply warped product $M = I\times_{h_1} N_1^{r_1} \times \cdots \times_{h_k} N_{k}^{r_k}$ satisfying the conventions of Remark \ref{grmnt_wrpngfnctn}. If $M$ has harmonic Weyl curvature and also admits a structure of almost Ricci soliton with nonconstant $f$, then $k \leq 2$.
\end{corollary}

\begin{proof}
This follows immediately from \cite[Corollary 2.6]{brgs1} and Lemma \ref{lem_ply_deg2_}.
\end{proof}

Another consequence of using \eqref{eqXiRelFuncWarp} is the following one.

\begin{corollary}\label{numbereigen}
	Let $(M^n,g,f,\lambda)$, $n\geq4$, be a gradient almost Ricci soliton with harmonic Weyl curvature and nonconstant $f$. There cannot be more than two distinct $\lambda_{i}$ with $i\in\{2,\ldots,n\}$. Consequently, the Ricci tensor has at most three eigenvalues.
\end{corollary}

Now we show that Corollary \ref{numbereigen} implies Theorem \ref{maxxnumeig-INTRO}.

\begin{proof}[{\bf Proof of Theorem \ref{maxxnumeig-INTRO}}]
	Let $(M^n,g,f,\lambda)$ be an almost Ricci soliton with $n\geq4$ and let $p\in M_{\mathcal{A}}$. It follows from Lemma \ref{denseempty} that there is a connected open set $U\subset M_{\mathcal{A}}$ containing $p$ such that the number of eigenvalues of $\mathrm{Ric}$ in $U$ is constant and $U\cap\mathcal{R}$ is either empty or dense in $U$. If it is empty, then the number of eigenvalues of $\mathrm{Ric}$ in $U$ is constant and equals one. If it is dense, Corollary \ref{numbereigen} asserts that $\mathrm{Ric}$ has exactly two or three eigenvalues in $U\cap\mathcal{R}$, and consequently, the same is true in $U$.
\end{proof}

As an immediate consequence:

\begin{proof}[{\bf Proof of Theorem \ref{decompwarpint}}]
	The proof of Theorem \ref{decompwarpint} is obtained immediately by combining Theorem \ref{decompwarp} and Corollary \ref{numbfibalmo}.
\end{proof}

	We finish this section with an example of a gradient almost Ricci soliton constructed on a multiply warped product with as many fibers as one wishes.

\begin{example}\label{exwmf}
	We shall construct a family of gradient almost Ricci solitons on multiply warped products with an arbitrary number of fibers (in the sense of Remark \ref{grmnt_wrpngfnctn}). Fix any integer $k \geq 2$, and for each $1 \leq i \leq k$, let $\mathcal{N}_i^{r_i}$ be a Ricci-flat manifold of (arbitrary) dimension $r_i$. Define 
	$$
n=\sum_{j=1}^{k} r_j + 1,\qquad 
C=\sum_{j=1}^{k} r_j j,\qquad 
L^2=(n-1)\sum_{j=1}^{k} r_j j^2 - C^2>0.
$$
Set $\varepsilon=\dfrac{\pi}{2L}$ and $I=\pa{-\varepsilon,\varepsilon}$. For $s\in I$, define
$$
\begin{aligned}
h_i(s)&=\exp\pa{\pa{i-\frac{C}{n-1}}s}\,\pa{\cos\pa{Ls}}^{-\frac{1}{n-1}}>0,\quad 1\le i\le k,\\
f(s)&=-\log\pa{\cos\pa{Ls}},\\
\lambda(s)&=-\dfrac{L^2}{\pa{n-1}\cos^2\!\pa{Ls}}.
\end{aligned}
$$
	A straightforward computation shows that, for all $1\leq i\leq k$ and $s \in I$,
$$
\begin{aligned}
f''(s)-\sum_{j=1}^{k}r_j \frac{h_j''(s)}{h_j(s)}&=\lambda(s),\\
-\frac{h_i''(s)}{h_i(s)}-\pa{r_i-1}\!\pa{\frac{h_i'(s)}{h_i(s)}}^{\!2}-\frac{h_i'(s)}{h_i(s)}\sum_{\substack{1 \leq j \leq k \\
    j \neq i}} r_j \frac{h_j'(s)}{h_j(s)} + f'(s) \frac{h_i'(s)}{h_i(s)}&=\lambda(s).
\end{aligned}
$$
Therefore,
$$
\bigl(\mathcal{M} = I \times_{h_1} \mathcal{N}_1^{r_1}\times_{h_2} \mathcal{N}_2^{r_2}\times\cdots\times_{h_k} \mathcal{N}_k^{r_k},\, g,\, f,\, \lambda\bigr)
$$
is a gradient almost Ricci soliton. By our very construction, the functions $h_i$ comply with the conventions of Remark \ref{grmnt_wrpngfnctn}.

\end{example}

\begin{remark}
	As one ought to expect from Theorem \ref{decompwarpint}, the example above does not have harmonic Weyl curvature. In fact, straightforward computations show that
	$$\frac{h_i''}{h_i}-\frac{h_j''}{h_j}
=\frac{i-j}{(n-1)^2} \Big(i+j-2C+2L\tan(Ls)\Big)$$
	for all $1\leq i,j\leq k$, thus the Cotton tensor is nonvanishing (see Remark \ref{rmk_mult_warp_more_fibers}). Even if one were to choose different constants in the definition of $h_i$, it is not possible to make all these expressions vanish simultaneously for $k\geq 3$ unless one forces $h_i \sim h_j$ (in the sense of Remark \ref{grmnt_wrpngfnctn}) for all but two $i,j$.
\end{remark}

\section{The almost Ricci soliton equations with at most two fibers}\label{subs_mult_warp}

\hspace{.5cm}In this section, we consider manifolds of the form $I\times_{h} N^{n-1}$ and $I\times_{h_1} N_1^{r_{1}} \times_{h_2} N_2^{r_{2}}$, and translate the gradient almost Ricci soliton equation and the vanishing of the Cotton tensor in terms of the warping functions and the geometry of the fibers. Recall that according to Theorem \ref{numbfibalmo}, we only need to consider these two cases. In our analysis, we will see that there are no further restrictions in the case of only one fiber (see Propositions \ref{pendT51_harmonic} and \ref{harmWarp} below). However, in the case of two fibers, the situation is rather more restrictive.

\subsection{Equations when there is only one fiber}\label{subs_mult_warp_1F}

\hspace{.5cm} In this subsection, we deal with the case where we have only one fiber. In the proposition below, whose proof is a straightforward computation using Lemma \ref{warped}, we rewrite the gradient almost Ricci soliton equation. 
\begin{proposition}\label{T31}
	Let $I\times_{h} N^{n-1}$ be a warped product. There are smooth functions $f,\lambda:I\rightarrow\mathbb{R}$ such that $(I\times_{h} N^{n-1}, g,f,\lambda)$ is an almost Ricci soliton if and only if $N^{n-1}$ is an Einstein manifold with Einstein constant $\mu$, and the functions $f,\lambda:I\rightarrow\mathbb{R}$ satisfy the system
	\begin{align}\label{SistemaSimples} 
		\begin{split}
			&\lambda = f'' - \frac{(n-1)}{h} h'' \vspace{0.3cm},\\
			&\lambda = f' \frac{h'}{h} + \frac{\mu}{h^2} - \frac{h''}{h} - \frac{(n-2)}{h^2} (h')^2 \vspace{0.3cm} .
		\end{split}
	\end{align}
\end{proposition}

A simple manipulation of equations \eqref{SistemaSimples} shows how the warping function can be used to determine both $f$ and $\lambda$. This is the content of the following proposition.
\begin{proposition}\label{pendT51_harmonic}
	The four-tuple $(I\times_{h} N^{n-1}, g,f,\lambda)$ is an almost Ricci soliton if and only if $N^{n-1}$ is an Einstein manifold with Einstein constant $\mu$, and the functions $f,\lambda:I\rightarrow\mathbb{R}$ are determined by $h$ through the following expressions
	\begin{align}
		&f'(s)=h(s)\int^{s}\frac{\mu+(n-2)(h(t)h''(t)-(h'(t))^2)}{(h(t))^3}\mathrm{d}t\label{eqitem_1_harm}\\
		&h\left(\frac{\lambda}{h'}\right)'=(n-1)\frac{h''}{h}+\mu\left(\frac{1}{hh'}\right)'-\left(\frac{h''}{h'}\right)'-(n-2)\left(\frac{h'}{h}\right)'\label{eqitem_2_harm}
	\end{align}
\end{proposition}
\begin{proof}
	From \eqref{SistemaSimples} we can write

	\begin{align*}
		\frac{h f'' - f'h'}{h^2} = \left( n-1 \right) \frac{h''}{h^2} + \frac{\mu}{h^{3}} - \frac{h'' }{h^{2} }  - \left( n-2 \right) \frac{\left( h' \right)^{2} }{h^{3} },  
	\end{align*}
	and consequently,
	\begin{align*}
		\left( \frac{f'}{h }  \right)' &= \left( n-2 \right) \left(\frac{ h''h - \left( h' \right)^{2}  } {h^{3}} \right) + \frac{\mu}{ h^{3}}  \\
		&= \frac{\mu + \left( n-2 \right) \left( h'' h - \left( h' \right)^{2}  \right) }{h^{3} }.
	\end{align*}
	The expression above is clearly equivalent to \eqref{eqitem_1_harm}. A straightforward computation shows that \eqref{eqitem_2_harm} is a consequence of the second equation in \eqref{SistemaSimples} combined with \eqref{eqitem_1_harm}.
\end{proof}

Now we consider the vanishing Cotton condition on a warped product with a single fiber. Straightforward computations show that in any multiply warped product with harmonic Weyl curvature, every fiber is Einstein, independently of the almost gradient Ricci soliton structure and regardless of the base dimension. In the one-fiber case, the converse holds: any warped product with a single Einstein fiber has a vanishing Cotton tensor; for a proof, see Example 16.26(1) in \cite{besse}.

\begin{proposition}\label{harmWarp}
	Let $M^n = I\times_{h} N^{n-1}$ be a warped product, where $N$ is an Einstein manifold. Then $M$ has harmonic Weyl curvature.
\end{proposition}
Putting together Proposition \ref{pendT51_harmonic} and Proposition \ref{harmWarp}, we have the following result

\begin{corollary}\label{cor_defin_1fiber}
	Let $I\times_{h} N^{n-1}$ be a warped product of an interval and an Einstein manifold $N^{n-1}$ having Einstein constant $\mu$. Set $f,\lambda:I\rightarrow\mathbb{R}$ using \eqref{eqitem_1_harm} and \eqref{eqitem_2_harm}. Then $(I\times_{h} N^{n-1}, g,f,\lambda)$ is an almost Ricci soliton having harmonic Weyl tensor.
\end{corollary}

If $N^{n-1}$ is taken with constant sectional curvature, then we have the following analogues:

\begin{proposition}\label{LCFWarp}
	Let $M^n = I\times_{h} N^{n-1}$ be a warped product, where $N$ has constant sectional curvature. Then $M$ is locally conformally flat.
\end{proposition}

\begin{corollary}\label{LCF_cor_defin_1fiber}
	Let $I\times_{h} N^{n-1}$ be a warped product of an interval and a manifold $N^{n-1}$ having constant sectional curvature $\mu/(n-2)$. Set $f,\lambda:I\rightarrow\mathbb{R}$ using \eqref{eqitem_1_harm} and \eqref{eqitem_2_harm}. Then $(I\times_{h} N^{n-1}, g,f,\lambda)$ is a locally conformally flat almost Ricci soliton.
\end{corollary}

We finish this subsection by observing that when there is only one fiber (counted according to Remark \ref{grmnt_wrpngfnctn}) and the gradient almost soliton is not Einstein, the number of distinct eigenvalues of the Ricci tensor is necessarily two.

\subsection{Equations when there are two fibers}\label{subs_mult_warp_2F}

\hspace{.5cm} In this subsection, we consider the case where we have two fibers. The proposition below rewrites the gradient almost Ricci soliton equation in this case. Its proof follows from a straightforward computation using Lemma \ref{warped}.

   \begin{proposition}\label{localstrucDupVBhr}
	Let $I\times_{h_1} N_1^{r_{1}} \times_{h_2} N_2^{r_{2}}$ be a multiply warped product. Then there exist functions $f,\lambda:I\rightarrow\mathbb{R}$ such that $(I \times_{h_1} N_1^{r_{1}} \times_{h_2} N_2^{r_{2}}, g,f,\lambda)$ is an almost Ricci soliton if and only if each fiber of dimension at least two is an Einstein manifold, with Einstein constant $\mu_{i}$, and the functions $f$ and $\lambda:I\rightarrow\mathbb{R}$ satisfy the system
	\begin{equation}\label{SistemaTriplo}
		\begin{split}
			&\lambda = f'' - r_1 \frac{h_1''}{h_1} - r_2 \frac{h_2''}{h_2}, \vspace{0.3cm} \\
			&\lambda = \frac{\mu_{1}}{(h_1)^2} + \frac{h_1'}{h_1} f' - \frac{h_1''}{h_1} - (r_1 - 1) \frac{(h_1')^2}{(h_1)^2} - r_2 \frac{h_1 ' h_2' }{h_1 h_2}, \vspace{0.3cm}\\ 
			&\lambda = \frac{\mu_{2}}{(h_2)^2} + \frac{h_2'}{h_2} f' - \frac{h_2''}{h_2} - (r_2 - 1) \frac{(h_2')^2}{(h_2)^2} - r_1 \frac{h_1 ' h_2' }{h_1 h_2}.
		\end{split}
	\end{equation}
\end{proposition}

It is convenient to make the following observations about a gradient almost Ricci soliton which is not Einstein, and is a multiply warped product with a one-dimensional base and exactly two fibers, counted according to Remark \ref{grmnt_wrpngfnctn}:
\begin{enumerate}
	\item The case $\lambda_{1}=\lambda_{2}$ does not happen, as by \eqref{eqXiRelFuncWarp} it would imply that $h_{1}$ is a constant multiple of $h_{2}$, violating Remark \ref{grmnt_wrpngfnctn};
	\item If the number of eigenvalues of the Ricci tensor is two, then the eigenspace corresponding to $\lambda_{1}$ is $(1+r_{1})$-dimensional, and is tangent to $I\times N_{1}^{r_{1}}$. In this case, the eigenspace corresponding to $\lambda_{2}$ is $r_{2}$-dimensional, and is tangent to $N_{2}^{r_{2}}$;
	\item If the number of eigenvalues of the Ricci tensor is three, then there is one eigenvalue distinct from $\lambda_{1}$ and $\lambda_{2}$, whose eigenspace is $1$-dimensional and is tangent to $I$. In this case, the eigenspace corresponding to $\lambda_{i}$ is $r_{i}$-dimensional, and is tangent to $N_{i}^{r_{i}}$.
\end{enumerate}

When we have two fibers, the condition of having harmonic Weyl curvature is not automatically satisfied, as it is when we have only one fiber. Thus, there is no result such as Proposition \ref{harmWarp} when there are two fibers. The results below will be used to analyze the case in which the Ricci tensor has exactly two eigenvalues, leading to Theorem \ref{pendT14'}.

In the rest of this subsection, we rewrite equations \eqref{SistemaTriplo} in a way that makes it easier to manipulate.

From now on, we are going to denote the eigenvalues of the Ricci tensor by $R_{11}$, $R_{22}$ and $R_{33}$. By the discussion above, we cannot have $\mathrm{Ric}_{22}\neq \mathrm{Ric}_{33}$. By making use of Lemma \ref{warped}, the eigenvalues of the Ricci tensor are given in terms of the warping functions and the dimensions of the fibers by
\begin{align}
	&\mathrm{Ric}_{11}=-r_{1}\frac{h''_{1}}{h_{1}}-r_{2}\frac{h''_{2}}{h_{2}}\label{neq1}\\
	&\mathrm{Ric}_{22}=-\frac{h''_{1}}{h_{1}}-(r_{1}-1)\left(\frac{h'_{1}}{h_{1}}\right)^2-r_{2}\frac{h'_{1}h'_{2}}{h_{1}h_{2}}+\frac{\mu_{1}}{h^{2}_{1}}\label{neq2}\\
	&\mathrm{Ric}_{33}=-\frac{h''_{2}}{h_{2}}-(r_{2}-1)\left(\frac{h'_{2}}{h_{2}}\right)^2-r_{1}\frac{h'_{1}h'_{2}}{h_{1}h_{2}}+\frac{\mu_{2}}{h^{2}_{2}}\label{neq3}
\end{align}

It is convenient to set the functions $a$ and $b$ as
\begin{align}\label{defab}
	a=\frac{h'_{1}}{h_{1}}\ \ \ \ \text{and}\ \ \ \ b=\frac{h'_{2}}{h_{2}}.
\end{align}
Observe that we cannot have $a=b$, as it would be equivalent to $\mathrm{Ric}_{22}=\mathrm{Ric}_{33}$, which is a contradiction.

The functions $a$ and $b$ allow to rewrite equations \eqref{neq1}, \eqref{neq2} and \eqref{neq3} as
\begin{align}
	&\mathrm{Ric}_{11}=-r_{1}(a'+a^2)-r_{2}(b'+b^2)\label{neq12}\\
	&\mathrm{Ric}_{22}=-a'-r_{1}a^2-r_{2}ab+\frac{\mu_{1}}{h^{2}_{1}}\label{neq22}\\
	&\mathrm{Ric}_{33}=-b'-r_{2}b^2-r_{1}ab+\frac{\mu_{2}}{h^{2}_{2}}\label{neq32}
\end{align}

With this notation, it is easy to prove the following proposition.
\begin{proposition}\label{prop71PEND} Let $(I\times_{h_1}N_1^{r_{1}} \times_{h_2}N_2^{r_{2}},g,f,\lambda)$ be a multiply warped product almost Ricci soliton, where $N_i$ is an Einstein manifold if $r_{i}\geq2$, with Einstein constants $\mu_{i}$. Then equation \eqref {fundeq} is equivalent to
	\begin{align}
		&f''=r_{1}(a'+a^2)+r_{2}(b'+b^2)+\lambda\label{neq14}\\
		&f'a=a'+r_{1}a^2+r_{2}ab-\frac{\mu_{1}}{h^{2}_{1}}+\lambda\label{neq24}\\
		&f'b=b'+r_{2}b^2+r_{1}ab-\frac{\mu_{2}}{h^{2}_{2}}+\lambda\label{neq34}
	\end{align}
\end{proposition}

When there are two fibers, the harmonicity of the Weyl tensor for gradient almost Ricci solitons is equivalent to the following condition.

\begin{proposition}
    Let $(I\times_{h_1}N_1^{r_{1}} \times_{h_2}N_2^{r_{2}},g,f,\lambda)$ be a multiply warped product almost Ricci soliton, where $N_i$ is an Einstein manifold if $r_{i}\geq2$, with Einstein constants $\mu_{i}$. Then $I\times_{h_1}N_1^{r_{1}} \times_{h_2}N_2^{r_{2}}$ has harmonic Weyl tensor if and only if
    \begin{align}\label{symmet}
        a'+a^2=b'+b^2,
    \end{align}
    or, equivalently, if and only if
    \begin{align}\label{symmet_hs}
        \frac{h''_{1}}{h_{1}}=\frac{h''_{2}}{h_{2}}.
    \end{align}
\end{proposition}
\begin{proof}
    We begin by recalling Lemma \ref{bari}, which states that $I\times_{h_1}N_1^{r_{1}} \times_{h_2}N_2^{r_{2}}$ has harmonic Weyl tensor if and only if the tensors $P$ and $Q$, defined respectively as
	\begin{align*}
	    P(X, Y, Z) = \mathrm{Rm
}(\nabla f, X, Y, Z)
	\end{align*}
	and
	\begin{align*}
	Q(X, Y, Z) = \frac{1}{n-1} \big(\mathrm{Ric}(\nabla f, Y) g(X, Z) - \mathrm{Ric}(\nabla f, Z) g(X, Y)\big),
	\end{align*}
	coincide, i.e., $P = Q$.
	Now, let $(x_\ell)_{1 \leq \ell \leq r_1}$ and $(y_\beta)_{1 \leq \beta \leq r_2}$ be local coordinate systems on the fibers $N_1$ and $N_2$, respectively. Using the standard curvature formulas for multiply warped products, the tensor $P$ can be expressed as:
	\begin{align*}
	P = (a' + a^2) \sum_{\ell=1}^{r_1} \mathrm{d} x^\ell \otimes (\mathrm{d} x^\ell \wedge \mathrm{d} t) 
	+ (b' + b^2) \sum_{\beta=1}^{r_2} \mathrm{d} y^\beta \otimes (\mathrm{d} y^\beta \wedge \mathrm{d} t).
	\end{align*}
	Similarly, the tensor $Q$ is given by:
	\begin{align*}
	Q = \frac{r_1 (a' + a^2) + r_2 (b' + b^2)}{n-1} 
	\left( \sum_{\ell=1}^{r_1} \mathrm{d} x^\ell \otimes (\mathrm{d} x^\ell \wedge \mathrm{d} t) 
	+ \sum_{\beta=1}^{r_2} \mathrm{d} y^\beta \otimes (\mathrm{d} y^\beta \wedge \mathrm{d} t) \right).
	\end{align*}
	Therefore, $P$ and $Q$ coincide if, and only if,
	\begin{align*}
        a' + a^2 &= \frac{r_1 (a' + a^2) + r_2 (b' + b^2)}{n-1}, \\
    	b' + b^2 &= \frac{r_1 (a' + a^2) + r_2 (b' + b^2)}{n-1}.
	\end{align*}
	or, equivalently, if and only if
	\begin{align*}
        a' + a^2 = b' + b^2.
	\end{align*}

\end{proof}

\begin{remark}\label{rmk_mult_warp_more_fibers}
	It is easy to check that the proof above naturally extends to the case of multiply warped products with more than two fibers. Namely, if $\ (I\times_{h_1}N_1^{r_{1}} \times_{h_2}N_2^{r_{2}}\times\cdots\times_{h_k}N_k^{r_{k}},g,f,\lambda)$ is a multiply warped product almost Ricci soliton, where $N_i$ is an Einstein manifold if $r_{i}\geq2$, with Einstein constants $\mu_{i}$, then $I\times_{h_1}N_1^{r_{1}} \times_{h_2}N_2^{r_{2}}\times\cdots\times_{h_k}N_k^{r_{k}}$ has harmonic Weyl tensor if and only if
	\begin{align*}
	    a'_i+a^2_i=a'_j+a^2_j
	\end{align*}
	for all $1\leq i,j\leq k$, where $a_i=h'_i/h_i$. 
\end{remark}

In the next lemma we use \eqref{symmet} to present $\lambda$ only in terms of the warping functions.

\begin{proposition}
    Let $(I\times_{h_1}N_1^{r_{1}} \times_{h_2}N_2^{r_{2}},g,f,\lambda)$ be a multiply warped product almost Ricci soliton, where $N_i$ is an Einstein manifold if $r_{i}\geq2$, with Einstein constants $\mu_{i}$. Assume it has harmonic Weyl tensor. Then the function $\lambda$ can be expressed in the following way
    \begin{align}\label{lambda_good}
        \lambda=-a'-r_{1}a^2-b'-r_{2}b^2+\frac{\mu_{1}}{h^{2}_{1}}+\frac{\mu_{2}}{h^{2}_{2}}
    \end{align}
\end{proposition}
\begin{proof}
To prove the proposition, first compute the difference between \eqref{neq24} and \eqref{neq34}, and use \eqref{symmet} to obtain
\begin{align*}
    (a-b)f'
    &=((r_{1}-1)a+(r_{2}-1)b)(a-b)-\frac{\mu_{1}}{h^{2}_{1}}+\frac{\mu_{2}}{h^{2}_{2}}.
\end{align*}
Now, taking the derivative on both sides of the equality above, and using \eqref{defab} and \eqref{symmet}, we deduce that
\begin{align}\label{frst_way}
\begin{split}
    ((a-b)f')'
            =&((r_{1}-1)a'+(r_{2}-1)b')(a-b)+2\frac{\mu_{1}}{h^{2}_{1}}a-2\frac{\mu_{2}}{h^{2}_{2}}b\\
    &-(a+b)(a-b)((r_{1}-1)a+(r_{2}-1)b).
\end{split}
\end{align}
On the other hand, insert \eqref{symmet} and \eqref{neq14} into $((a-b)f')'=(a'-b')f'+(a-b)f''$ and then reorganize the terms, to get
\begin{align}\label{scnd_way}
\begin{split}
    ((a-b)f')'
    =&-(a+b)(a-b)((r_{1}-1)a+(r_{2}-1)b)+(a+b)\left(\frac{\mu_{1}}{h^{2}_{1}}-\frac{\mu_{2}}{h^{2}_{2}}\right)\\
    &+(a-b)(r_{1}a'+r_{2}b')+(a-b)(r_{1}a^2+r_{2}b^2+\lambda)
\end{split}
\end{align}
Finally, equate \eqref{frst_way} and \eqref{scnd_way}, and eliminate the similar terms, to conclude that
\begin{align*}
    (a-b)\left(\frac{\mu_{1}}{h^{2}_{1}}+\frac{\mu_{2}}{h^{2}_{2}}-a'-b'\right)=(a-b)(r_{1}a^2+r_{2}b^2+\lambda).
\end{align*}
Use that $a\neq b$ to cancel out the difference $a-b$ in both sides, and this proves the proposition.
\end{proof}

\section{Harmonic Weyl tensor: exactly two eigenvalues}\label{quantity_multip}

\hspace{.5cm} In this subsection we prove Theorem \ref{pendT14'}. We are assuming that $n\geq4$ and that the Weyl tensor of the gradient almost Ricci soliton is harmonic. In this case, the Cotton tensor vanishes, which is equivalent to the Schouten tensor being Codazzi. Thus, we can use all results of Section \ref{cotton_results}. We are also considering the case where the Ricci tensor has exactly two eigenvalues, that is $\mathrm{Ric}_{11}\in\{\mathrm{Ric}_{22},\mathrm{Ric}_{33}\}$, with $\mathrm{Ric}_{22}\neq \mathrm{Ric}_{33}$. Without loss of generality, assume that $\mathrm{Ric}_{11}=\mathrm{Ric}_{22}$. In this case, we have the following lemma.
\begin{lemma}\label{f_determines_ever}
    Assume that $\mathrm{Ric}_{11}=\mathrm{Ric}_{22}$. Then:
    \begin{enumerate}
        \item\label{f_determines_ever1} There is a constant $C_{1}\neq0$ such that $h_{1}=C_{1}^{-1}f'$;
        \item\label{f_determines_ever2} $\lambda=f''-(n-1)\frac{f'''}{f'}$,
        \item\label{f_determines_ever3} If $h_{2}$ is not constant, then there are $C_{2},C_{3}\in\mathbb{R}$ such that $h_{2}=-(n-1)C_{1}^{-1}e^{-\frac{1}{n-1}(f-C_{2})}+C_{3}$.
    \end{enumerate}
\end{lemma}
\begin{proof}
    First of all, notice that the assumption $\mathrm{Ric}_{11}=\mathrm{Ric}_{22}$, combined with \eqref{neq12}-\eqref{neq22} and \eqref{neq14}-\eqref{neq24}, implies immediately that $f''=af'$. Now, using the definition of $a$, we can write
\begin{align}\label{frst_Key}
    \frac{f''}{f'}=\frac{h'_{1}}{h_{1}},
\end{align}
implying the existence of $C_{1}\neq0$ such that $f'=C_{1}h_{1}$, and then proving item \ref{f_determines_ever1}. To see that item \ref{f_determines_ever2} is true, just insert item \ref{f_determines_ever1} into the first equation of \eqref{SistemaTriplo} to get
\begin{align}
    \lambda=f''-(n-1)\frac{h''_{1}}{h_{1}}=f''-(n-1)\frac{f'''}{f'},
\end{align}
where we have used \eqref{symmet_hs} and $r_1+r_2=n-1$ in the first equality.

To the proof of item \ref{f_determines_ever3}, notice that putting $\mathrm{Ric}_{11}=\mathrm{Ric}_{22}$ together with \eqref{neq12}-\eqref{neq22} implies
\begin{align*}
    -(n-1)(b'+b^2)+r_{2}ab=-a'-r_{1}a^2+\frac{\mu_{1}}{h^{2}_{1}},
\end{align*}
which can be coupled with \eqref{lambda_good}, giving rise to
\begin{align*}
    \lambda=-(r_{1}+r_{2}+1)b'-(r_{1}+2r_{2})b^2+r_{2}ab+\frac{\mu_{2}}{h^{2}_{2}}.
\end{align*}
Now, using the equation above to simplify \eqref{neq34}, we obtain the equation
\begin{align*}
    f'b=(n-1)(-b'-b^2+ab).
\end{align*}
Finally, assume that $h_{2}$ is not constant. Thus, we may divide the equality above by $b$ to get
\begin{align*}
    f'=(n-1)(-\frac{b'}{b}-b+a)=-(n-1)\left(\ln\left(\frac{h_{2}b}{h_{1}}\right)\right)'=-(n-1)\left(\ln\left(\frac{h'_{2}}{h_{1}}\right)\right)',
\end{align*}
where we have used the definitions of $a$ and $b$. As a consequence, there is a constant $C_{2}$ such that
\begin{align*}
    f=-(n-1)\ln\left(\frac{h'_{2}}{h_{1}}\right)+C_{2}.
\end{align*}
The identity above gives the following expression for $h_{2}'$,
\begin{align*}
    h'_{2}=h_{1}e^{-\frac{1}{n-1}(f-C_{2})}=C_{1}^{-1}f'e^{-\frac{1}{n-1}(f-C_{2})}=-(n-1)C_{1}^{-1}\left(e^{-\frac{1}{n-1}(f-C_{2})}\right)',
\end{align*}
where in the second equality we have used item \ref{f_determines_ever1}. This gives a constant $C_{3}$ such that
\begin{align*}
        h_{2}=-(n-1)C_{1}^{-1}e^{-\frac{1}{n-1}(f-C_{2})}+C_{3},
\end{align*}
proving item \ref{f_determines_ever3} and finishing the proof of the lemma.
\end{proof}

Now we use these representations for $h_{1},\ h_{2}\ \text{and}\ \lambda$ in terms of $f$ and its derivatives to obtain two differential equations for $f$.

\begin{lemma}
Assume that $\mathrm{Ric}_{11}=\mathrm{Ric}_{22}$. If $h_{2}$ is not constant, then $f$ must satisfy the following equations:
\begin{align}
    (n-1)f'''\Lambda&=(n-1)f'f''-(f')^3\label{sit_f_1'}\\
    ((n-2)f'f'''-(r_{1}-1)(f'')^2+C_{1}^{2}\mu_{1})\Lambda&=r_{2}(f')^2f''\label{sit_f_2'}
\end{align}
where the real constants $C_{1},\ C_{2}$ and $C_{3}$ are the same as in Lemma \ref{f_determines_ever}, and
\begin{align}
    \Lambda=C_{1}C_{3}e^{\frac{1}{n-1}(f-C_{2})}-(n-1).
\end{align}
\end{lemma}
\begin{proof}
To prove this lemma, we will use the expressions for $h_{1},\ h_{2}\ \text{and}\ \lambda$, given in Lemma \ref{f_determines_ever}.

To get \eqref{sit_f_1'}, consider the harmonic Weyl condition \eqref{symmet_hs}. In this case:
\begin{align*}
	\frac{f'''}{f'}=\frac{h''_{1}}{h_{1}}=\frac{h''_{2}}{h_{2}}=\frac{C^{-1}_{1}(f''-\frac{1}{n-1}(f')^2)e^{-\frac{1}{n-1}(f-C_{2})}}{-C_{1}^{-1}(n-1)e^{-\frac{1}{n-1}(f-C_{2})}+C_{3}}=\frac{f''-\frac{1}{n-1}(f')^2}{C_{1}C_{3}e^{\frac{1}{n-1}(f-C_{2})}-(n-1)}.
\end{align*}
Using the definition of $\Lambda$, given in the statement of the lemma, we obtain \eqref{sit_f_1'}.

Now we prove \eqref{sit_f_2'}. Notice that \eqref{neq24} can be written as $f'a=a'+a^2+(r_{1}-1)a^2+r_{2}ab-\frac{\mu_{1}}{h^{2}_{1}}+\lambda$. On the other hand, it is not hard to see that
\begin{align*}
    a=\frac{f''}{f'},\ \ a'+a^2=\frac{f'''}{f'}\ \ \text{and}\ \ b=\frac{f'}{\Lambda}.
\end{align*}
Putting all this together, we get
\begin{align*}
f''=\frac{f'''}{f'}+(r_{1}-1)\frac{(f'')^2}{(f')^2}+r_{2}\frac{f''}{\Lambda}-\frac{\mu_{1}C_{1}^{2}}{(f')^2}+f''-(n-1)\frac{f'''}{f'}.
\end{align*}
Canceling out $f''$, multiplying this equation by $(f')^2\Lambda$ and rearranging the terms, we get \eqref{sit_f_2'}.

\end{proof}

Now we treat the solutions of system \eqref{sit_f_1'}-\eqref{sit_f_2'}. It is not hard to see that constant functions solve this system. The next lemma, whose proof can be found in the appendix, asserts that these are the only solutions.

\begin{proposition}\label{constancy}
    Any function satisfying \eqref{sit_f_1'} and \eqref{sit_f_2'} simultaneously must be constant.
\end{proposition}

Now we are ready to prove Theorem \ref{pendT14'}.

\begin{proof}[{\bf Proof of Theorem \ref{pendT14'}}]
    Assume the gradient almost Ricci soliton $(M,g,f,\lambda)$ is nontrivial, has harmonic Weyl tensor and that its Ricci curvature has exactly two eigenvalues. Now, Theorem \ref{decompwarpint} ensures that around a regular point of $f$, $M$ must be a warped product with a one-dimensional base, and either one or two fibers. We consider each case separately.
    
    In the case of only one fiber, the gradient almost Ricci soliton must be as those described in Corollary \ref{cor_defin_1fiber} around a regular point, and this finishes the analysis in this case. 
    
    Now consider the case of exactly two fibers, and assume by contradiction that $h_{2}$ is not constant. In this case, $h_{1}$, $h_{2}$ and $\lambda$ must be given as in Lemma \ref{f_determines_ever}, and $f$ must satisfy equations \eqref{sit_f_1'} and \eqref{sit_f_2'}. But according to Proposition \ref{constancy}, $f$ must be constant, and we get a contradiction with the nontriviality of the gradient almost Ricci soliton. As a consequence, $h_{2}$ must be constant, say $h_{2}\equiv C$. In this case, using \eqref{symmet_hs} we obtain $h''_{1}\equiv0$, and hence there are constants $A,B\in\mathbb{R}$ such that
    \begin{align}\label{wrpfct_lin}
        h_{1}(s)=As+B.
    \end{align}
    Putting these warping functions in \eqref{SistemaTriplo}, we conclude that
    \begin{align}
			&f''=\lambda, \label{SistemaTriplo'1}\\
			&f'\frac{A}{As+B}=-\frac{A^2}{(As+B)^2}+r_{1}\frac{A^2}{(As+B)^2}-\frac{\mu_{1}}{(As+B)^2}+\lambda,\label{SistemaTriplo'2} \\ 
			&0=-\frac{\mu_{2}}{C}+\lambda.\label{SistemaTriplo'3}
	\end{align} 
    From \eqref{SistemaTriplo'3}, we conclude that $\lambda$ is constant, and from \eqref{SistemaTriplo'1}, we obtain constants $D,E\in\mathbb{R}$ such that
    \begin{align}\label{exp_for_f}
        f(s)=\frac{\lambda}{2}s^2+Ds+E.
    \end{align}
    Now, rewrite \eqref{SistemaTriplo'2} in the following way
    \begin{align*}
        \frac{(\lambda s+D)A}{As+B}=\frac{(r_{1}-1)A^2-\mu_{1}}{(As+B)^2}+\frac{\mu_{2}}{C}.
    \end{align*}
    Using \eqref{SistemaTriplo'3} and that $s$ varies on an open interval, the equality above gives
    \begin{align}\label{step_mid}
        AD-\lambda B=(r_{1}-1)A^2-\mu_{1}.
    \end{align}
    On the other hand, using $\mathrm{Ric}_{22}=\mathrm{Ric}_{11}=0$, \eqref{neq1} and \eqref{neq2}, we conclude that 
     \begin{align}\label{otherrelmu}
        \mu_{1}=(r_{1}-1)A^2.
    \end{align}
    Consequently, from \eqref{step_mid}, we get
    \begin{align}\label{otherrel}
        \lambda B=AD.
    \end{align}
    Observe that we cannot have $A=0$. In fact, if this were the case, then we would have $B=h_{1}>0$, and from \eqref{otherrel} that $\lambda=0$. Thus, the soliton equation would force $\mu_{2}=0$, and then $\mathrm{Ric}_{33}=0$, contradicting the assumption of exactly two distinct eigenvalues for the Ricci tensor of $M$. Thus, we have from \eqref{otherrelmu} that $\mu_{1}=(r_{1}-1)A^2>0$, if $r_{1}>1$.

    By Lemma \ref{warped}, we can see that $B^{r_{1}+1}=I \times_{h_1} N_1^{r_{1}}$ is Ricci-flat, where $h_{1}$ is given by \eqref{wrpfct_lin}. Also notice that $f$ given as in \eqref{exp_for_f} satisfies $\nabla_{B}\nabla_{B}f=\lambda g_{B}$. Normalizing $C$ to have $h_{2}=1$, we conclude the proof of Theorem \ref{pendT14'}.
\end{proof}

\begin{proof}[{\bf Proof of Corollary \ref{cor}}]
    Assume the gradient almost Ricci soliton has harmonic Weyl tensor, its Ricci tensor has exactly two eigenvalues and that $W(\nabla f,\cdot,\cdot,\cdot)$ does not vanish identically. Then we can apply Theorem \ref{pendT14'}, and the condition on the non-vanishing of $W(\nabla f,\cdot,\cdot,\cdot)$ means that $M$ is as described in item \ref{item2'}, namely, that $M$ is isometric to $B^k\times N^{n-k}$, where $(B^{k},g_{B})$ is a complete manifold carrying a vector field $\nabla_{B} f$ satisfying $\nabla_{B}\nabla_{B} f=\lambda g_{B}$ with $\lambda\neq0$ constant.
    
    To finish the proof we may invoke \cite[Theorémè 3]{kerbrat}, which tells us that the universal covering of $B^{k}$ must be isometric to $\mathbb{R}^k$ and that $f(x,p)=\frac{1}{2}\lambda \vert x\vert^2+\left\langle v,x\right\rangle+c$, for some $v\in\mathbb{R}^k$ and $c\in\mathbb{R}$ (see also \cite[Theorem A.5]{brgs1}).
    
    Alternatively, the proof of Theorem \ref{pendT14'}, given above, establishes that $B^{r_{1}+1}=I \times_{h_1} N_1^{r_{1}}$, with $h_{1}$ and $f$ given by \eqref{wrpfct_lin} and \eqref{exp_for_f}, respectively. Notice that $\{-A/B-\varepsilon\}\times N_{1}$ is at a finite distance $\vert s_{0}-\varepsilon\vert$ of the set $\{-A/B-s_{0}\}\times N_{1}$, for any $\varepsilon>0$ and a fixed $s_{0}>-A/B$. Using that $B^{r_{1}+1}$ is complete and that $\displaystyle\lim_{s\rightarrow-A/B}h_{1}(s)=0$, we conclude that $\{-A/B\}\times N_{1}$ corresponds to a point $p_{0}$ in the interior of $B^{r_{1}+1}$. Computing the sectional curvature of $B^{r_{1}+1}$ around this point, we get for any $U,V\in\mathfrak{L}(N_{1})$ that
    \begin{align*}
        \mathrm{sec}_B(U,V)=\frac{\mathrm{sec}_{N_{1}}(U,V)-A^2}{(As+B)^2}.
    \end{align*}
    Since the curvature of $B^{r_{1}+1}$ must be bounded around $p_{0}$, we conclude that $\mathrm{sec}_B(U,V)$ is finite when $s$ goes to $-B/A$, implying that $\mathrm{sec}_{N_{1}}(U,V)=A^2$ on $N_{1}$, for any $U,V\in\mathfrak{L}(N_{1})$, showing that $N_{1}$ must have constant positive sectional curvature. This shows that the universal covering of $B^{k}$ is $\mathbb{R}^k$, $k=r_{1}+1$, completing the proof. Now we go from polar coordinates to the Cartesian ones, transforming $f$ given by \eqref{exp_for_f}, into $f(x,p)=\frac{1}{2}\lambda \vert x\vert^2+\left\langle v,x\right\rangle+c$, for some $v\in\mathbb{R}^k$ and $c\in\mathbb{R}$.
\end{proof}

\begin{proof}[{\bf Proof of Corollary \ref{pendT15}}]
    Assume the gradient almost Ricci soliton is locally conformally flat. Catino showed in \cite{catino0} that the Ricci tensor of $M$ has exactly two eigenvalues, and that $M$ is locally given by $I\times_{h} N^{n-1}$, where $I$ is an interval and $N^{n-1}$ is a manifold of constant sectional curvature $\mu/(n-2)$. Arguing as in the proof of Proposition \ref{pendT51_harmonic}, we conclude that $f$ and $\lambda$ must be given as in \eqref{eqitem_1_harm} and \eqref{eqitem_2_harm}, respectively. The converse follows from Corollary \ref{LCF_cor_defin_1fiber}.
\end{proof}

\begin{proof}[{\bf Proof of Theorem \ref{scht_weyHar}}]
    By definition \cite{brgs3,catino}, a gradient Schouten soliton is obtained by setting $\lambda=(1/2(n-1))R+\tau$ in \eqref{fundeq}, where $R$ is the scalar curvature of $M^n$ and $\tau\in\mathbb{R}$. Since we are assuming that it has harmonic Weyl tensor, we may apply Theorem \ref{decompwarpint} to get local representations around regular points of $f$ as a multiply warped product with a one-dimensional base and at most two fibers.
    
    First, consider the case in which $M$ is locally given by $I\times_{h}N^{n-1}$, where $N^{n-1}$ is Einstein, with Einstein constant $\mu$. It follows from \eqref{bari} and \eqref{eqXiRelFuncWarp} that $h''=h(\xi'+\xi^{2})=0$. Therefore, there are constants $A,B\in\mathbb{R}$ such that $h(s)=As+B$. Using \eqref{SistemaSimples}, we conclude that the gradient Schouten soliton equation is equivalent to
    \begin{align*} 
	\begin{split}
		\frac{\mu-(n-2)A^2}{2(As+B)^2}+\tau = f''\ \ \ \text{and}\ \ \ \frac{\mu-(n-2)A^2}{2(As+B)^2}+\tau = \frac{A}{As+B}f' + \frac{\mu-(n-2)A^2}{(As+B)^2},
	\end{split}
    \end{align*}
    where we have already used the expression of the scalar curvature given in Lemma \ref{warped} to compute the function $\lambda=(1/2(n-1))R+\tau$. Integrating the equation for $f''$, we easily see that it is equivalent to the other one, which after integrating again, shows the existence of constants $c_{0},c_{1}\in\mathbb{R}$ such that
    \begin{align}
        f(s)=-\frac{\mu-(n-2)A^2}{2A^2}\log(As+B)+\frac{\tau}{2}s^2+c_{0}s+c_{1}.
    \end{align}
    If we assume that $M$ is complete, then we must have $\mu-(n-2)A^2=0$, otherwise $f$ would explode at an interior point of $M$. Arguing as in the proof of Corollary \ref{cor}, we conclude that $M$ is the Gaussian soliton.
    
    Now, consider the case where $M$ is locally given by $I\times_{h_{1}}N_{1}^{r_1}\times_{h_{2}}N_{2}^{r_2}$, where $N^{r_i}_{i}$ is Einstein if $r_i\geq2$, with Einstein constant $\mu_{i}$, $i\in\{1,2\}$. Thus, there are constants $A_{i},B_{i}\in\mathbb{R}$ such that $h_{i}(s)=A_{i}s+B_{i}$. Using equations \eqref{SistemaTriplo}, we conclude that the gradient Schouten soliton equation becomes
    \begin{align}
        \lambda=f''\ \ \ \text{and}\ \ \ 
        \lambda=\frac{A_{i}}{A_{i}s+B_{i}}f'+\frac{\mu_{i}-(r_{i}-1)A_{i}^{2}}{(A_{i}s+B_{i})^2}-\frac{r_{i}A_{1}A_{2}}{(A_{1}s+B_{1})(A_{2}s+B_{2})},\ \ \ i\in\{1,2\}.
    \end{align}

    In what follows, we prove that either $A_{1}=0$ or $A_{2}=0$. In this case, we may consider without loss of generality that $A_{1}\neq0$ and $A_{2}=0$. Now, it is sufficient to proceed as in the proofs of Theorem \ref{pendT14'} and Corollary \ref{cor} to conclude that $M$ is rigid, as desired.
    
    To prove that $A_{1}A_{2}=0$, first notice that subtracting both equations involving $f'$ from one another and solving for $f'$, we obtain
    \begin{align}\label{eq1_mim}
        f'=\frac{\mu_{2}-(r_{2}-1)A_{2}^{2}}{A_{1}B_{2}-A_{2}B_{1}}\left(\frac{A_{1}s+B_{1}}{A_{2}s+B_{2}}\right)+\frac{\mu_{1}-(r_{1}-1)A_{1}^{2}}{A_{1}B_{2}-A_{2}B_{1}}\left(\frac{A_{2}s+B_{2}}{A_{1}s+B_{1}}\right)-\frac{(r_{1}-r_{2})A_{1}A_{2}}{A_{1}B_{2}-A_{2}B_{1}},
    \end{align}
    where we are using that $A_{1}B_{2}-A_{2}B_{1}\neq0$, which is equivalent to $h_1$ not being a constant multiple of $h_{2}$. Now, using $\lambda=f''$ and the expression of the scalar curvature given in Lemma \ref{warped}, we conclude that
    \begin{align*}
        f''=\frac{1}{2(n-1)}\left(\frac{r_{1}(\mu_{1}-(r_{1}-1)A_{1}^{2})}{(A_{1}s+B_{1})^2}+\frac{r_{2}(\mu_{2}-(r_{2}-1)A_{2}^{2})}{(A_{2}s+B_{2})^2}-\frac{2r_{1}r_{2}A_{1}A_{2}}{(A_{1}s+B_{1})(A_{2}s+B_{2})}\right)+\tau.
    \end{align*}
    
    Finally, assume by contradiction that $A_{1}A_{2}\neq0$. Then we can integrate the equation above, obtaining
    \begin{align*}
        f'=-\frac{1}{2(n-1)}\left(\frac{r_{1}(\mu_{1}-(r_{1}-1)A_{1}^{2})}{A_{1}(A_{1}s+B_{1})}+\frac{r_{2}(\mu_{2}-(r_{2}-1)A_{2}^{2})}{A_{2}(A_{2}s+B_{2})^2}+\frac{2r_{1}r_{2}A_{1}A_{2}}{A_{1}B_{2}-A_{2}B_{1}}\log\left(\frac{A_{1}s+B_{1}}{A_{2}s+B_{2}}\right)\right)+\tau.
    \end{align*}
    Exploring the fact that the equation above has a log term, while \eqref{eq1_mim} does not, we get a contradiction. This ensures that $A_{1}A_{2}=0$, finishing the proof.
\end{proof}

\section{Appendix: functions satisfying \eqref{sit_f_1'} and \eqref{sit_f_2'} are constant}\label{Ap_1}

\hspace{.5cm}In this appendix we prove Proposition \ref{constancy}, which asserts that a function satisfying \eqref{sit_f_1'} and \eqref{sit_f_2'} simultaneously must be constant. For the sake of completeness, we recall the equations in what follows.

Fix constants $C_{1}\neq0,\ C_{2}$ and $C_{3}$, and consider the equations
\begin{align}
    &(n-1)f'''\Lambda=(n-1)f'f''-(f')^3\tag{\ref{sit_f_1'}}\\
    &((n-2)f'f'''-(r_{1}-1)(f'')^2+C_{1}^{2}\mu_{1})\Lambda=r_{2}(f')^2f''\tag{\ref{sit_f_2'}}
\end{align}
where $\Lambda(f)=\Lambda=C_{1}C_{3}e^{\frac{1}{n-1}(f-C_{2})}-(n-1)$.

These equations are simultaneously satisfied by any constant when $\mu_1 = 0$, and admit a constant solution (the constant defined by the equation $\Lambda=0$) when $\mu\neq0$.  We will now employ a proof by contradiction to establish that these are the only functions satisfying both equations simultaneously. The case $r_1 = 1$ can be handled immediately, as shown below. 

\begin{proposition}\label{r1eq1}
    If $r_{1}=1$, then any function satisfying \eqref{sit_f_1'} and \eqref{sit_f_2'} simultaneously must be constant.
\end{proposition}
\begin{proof}
	Assume $r_1=1$ and, for the sake of contradiction, that there exists a nonconstant function $f$ satisfying \eqref{sit_f_1'} and \eqref{sit_f_2'}. Since $r_1 = 1$ implies $\mu_1=0$, it follows that $f$ satisfies
	\[
	(n-2)(f')^3f''=0 \quad\text{and}\quad 4(n-1)(n-2)(f')^2f''=(n-2)(f')^4.
	\]
	These identities force $f'\equiv0$, contradicting the assumption that $f$ is nonconstant.
\end{proof}

The case where $r_{1}>1$ is more involved. Assuming that there is a nonconstant function $f$  satisfying \eqref{sit_f_1'} and \eqref{sit_f_2'}, we will show the existence of a polynomial of degree $12$ for which $f'(s)$ is a solution, for any $s$. Then, this fact will be used to show the constancy of $f$, obtaining a contradiction (see Proposition \ref{propo_r1>1} below). The first step in obtaining this polynomial is given in the next result. It consists in finding other equations for a nonconstant function $f$ satisfying \eqref{sit_f_1'} and \eqref{sit_f_2'}, that do not involve $\Lambda$ or $f'''$.

\begin{lemma}
A nonconstant function $f$ satisfying \eqref{sit_f_1'} and \eqref{sit_f_2'} must also satisfy the following ODEs:
\begin{subequations}\label{wttLamb}
\begin{align}
&[(n-2)(f')^2 - (n-1)(r_1-1)f'']f'f''' =[C_1^2\mu_1 - (r_1-1)(f'')^2][(n-1)f'' - (f')^2],\label{wttLamb_1}\\\medskip
&\beta_1(f')^2f'' + \beta_3[(f'')^2 - f'f''']=\beta_2(f')^4 +(n-1)\beta_4\label{wttLamb_2},\\\medskip
&3 \beta_2\beta_3(f'')^2 - 4(n-1)\beta_2^2(f')^2f'' + \beta_2^2(f')^4 + (n-1)r_2 \beta_4 = 0, \label{wttLamb_3}
\end{align}
\end{subequations}
where 
\[ \begin{aligned}
& \beta_1 = (n-1)[3(r_1-1)+4r_2],\ \ \beta_2 = n-2,\ \ \beta_3 = (n-1)^2(r_1-1),\ \ \beta_4 = C_1^2 (n-1) \mu_1.
\end{aligned}
\]
\end{lemma}
\begin{proof}

	To obtain \eqref{wttLamb_1}, we first divide \eqref{sit_f_1'} by \eqref{sit_f_2'} and rearrange the resulting expression. Next, we observe that \eqref{sit_f_1'} can be used to eliminate $f'''$ from \eqref{sit_f_2'}. More precisely, setting $ (n-1)\beta_5 = \beta_3$, multiplying \eqref{sit_f_2'} by $(n-1)$ and substituting the expression for $f'''$ from \eqref{sit_f_1'} then leads to
	\begin{align}\label{new_wittLamb_1}
	\begin{split}
	[\beta_5 (f'')^2 -\beta_4 ]\Lambda 
	 &= -(n-1)r_{2}(f')^2f'' + (n-1)\beta_2 f'f'''\Lambda\\
	 &= -(n-1)r_{2}(f')^2f'' + \beta_2 f'[(n-1)f'f'' - (f')^3]\\
	 &= (n-1)[-r_{2} + \beta_2](f')^2f'' - \beta_2 (f')^4\\
	 &= \beta_5(f')^2f'' - \beta_2 (f')^4.
	\end{split}
	\end{align}
We now find another equation by taking the derivative of \eqref{new_wittLamb_1}. We will need the following identity:
	\begin{align}\label{lamb'}
	 (n-1)\Lambda' = f'\,(\Lambda + n - 1).
	\end{align}
	A straightforward computation shows that the derivative of \eqref{new_wittLamb_1} is
	\begin{align*}
	 [\beta_5(f'')^2 - \beta_4 ]\Lambda' 
	 &= -2\beta_5f''f'''\,\Lambda + \beta_5[2f'(f'')^2 + (f')^2 f''']-4\beta_2 (f')^3f''.
	\end{align*}
	Applying identity \eqref{lamb'} and \eqref{sit_f_1'} to replace $f'''$, we deduce that
	\begin{align*}
	[(r_{1}-1)(f'')^2 - C_1^2 \mu_1] f'\,(\Lambda + n - 1)
	 &= -2(r_{1}-1)f''[(n-1)f'f'' - (f')^3] \\
	 &\quad + \beta_5[2f'(f'')^2 + (f')^2 f''']- 4\beta_2 (f')^3f''.
	\end{align*}
	Reorganizing terms yields
	\begin{align*}
	[(r_{1}-1)(f'')^2 - C_1^2\,\mu_1] f'\,\Lambda
	 &= - \beta_5f'(f'')^2 + 2(r_{1}-1)(f')^3f''+ \beta_5(f')^2 f''' - 4\beta_2 (f')^3f'' + \beta_4 f'.
	\end{align*}
	After canceling $f'$ out, we arrive at
	\begin{align}\label{new_wittLamb_2}
	[(r_{1}-1)(f'')^2 - C_1^2\,\mu_1]\Lambda
	 = \beta_5[f'f'''-(f'')^2] - 2(r_{1}+2r_{2}-1)(f')^2f'' + \beta_4.
	\end{align}
	Now, notice that multiplying \eqref{new_wittLamb_2} by $n-1$, its left-hand side becomes the same as in \eqref{new_wittLamb_1}. Comparing both equations we get
	\[
	 \beta_3[f'f'''-(f'')^2] - 2(n-1)(r_{1}+2r_{2}-1)(f')^2f'' + (n-1)\beta_4
	 = \beta_5(f')^2f'' - \beta_2 (f')^4.
	\]
	Rearranging the terms, we obtain
	\[
	 \beta_3[f'f'''-(f'')^2] = \beta_1 (f')^2f'' - \beta_2 (f')^4 - (n-1)\beta_4,
	\]
	which is exactly \eqref{wttLamb_2}. In a similar way, by isolating $f'f'''$ in \eqref{wttLamb_2} and substituting back into \eqref{wttLamb_1}, we obtain \eqref{wttLamb_3}.
\end{proof}

In the next result we establish a polynomial $P(x)$ for which $P\bigl(f'(s)\bigr)=0,\ \forall s$. This will be used in the following proposition to show that \eqref{sit_f_1'}--\eqref{sit_f_2'} do not have any nonconstant solutions.

\begin{lemma}\label{lem:nontrivial_polynomial}
	Suppose that $r_1>1$. Then there exists a nonconstant polynomial 
	\[
	P(x)=a_{12}x^{12}+a_8x^8+a_4x^4+a_0,
	\]
	with coefficients $a_i$ depending only on $n$, $r_1$, $r_2$, and $C_1$, and for which $a_{12}\neq0$, such that 
	\[
	P\bigl(f'(s)\bigr)=0,\ \forall s.
	\]
\end{lemma}
    
\begin{proof}
    In the course of the proof, we will introduce a number of auxiliary coefficients, in order to simplify the argument.
    
	Set $y=f'$ and write $f''$ as a function of $y$, namely $f''=u(y)=u(f')$. We will represent the derivative of $u$ with respect to $y$ by $\dot{u}$. With this notation, equation \eqref{wttLamb_2} then becomes 
			\begin{equation}\label{i}
				y u \dot{u} = u^2 + \tilde{\beta_1} y^2 u - \tilde{\beta_2} y^4 - \delta, \tag{i}
			\end{equation}
			where 
			\[
			\widetilde{\beta}_i\,\beta_3 = \beta_i \quad\text{and}\quad \delta = C_1^2 (n-1)^2 \mu_1.
			\]
		Now, setting
    \begin{align*}
        \gamma_{1}=3\beta_{1}\beta_{2} ,\ \ \gamma_{2}=4(n-1)\beta_{1}^2,\ \ \gamma_{3}=\beta^{2}_{2},\ \ \gamma_{4}=(n-1)r_{2}\beta_{4},
    \end{align*}
		we may rewrite \eqref{wttLamb_3} in the equivalent form
		\[\gamma_1 u^2 - \gamma_2 y^2 u  + \gamma_3 y^4 + \gamma_4 = 0.\]
		One can then solve this quadratic equation for $u$ as a function of $y$, obtaining
		\begin{equation}\label{ii}
	      u(y) = \eta_1 y^2 + \varepsilon \sqrt{\eta_2 y^4 - \eta_3}\tag{ii}.
	    \end{equation}
		where $\epsilon \in \{-1, 1\}$ and the constant $\eta_{i}'$s are defined by
        \begin{align*}
             2\gamma_{1}\eta_{1}=\gamma_{2},\ 4\gamma_{1}^2\eta_{2}=\gamma_{2}^2-4\gamma_{1}\gamma_{3},\ \gamma_{1}\eta_{3}=\gamma_{4}.
        \end{align*}
	This immediately implies that
				\begin{equation}\label{iii}
	\dot{u}
	=2\,\eta_{1}y
	+\frac{2\varepsilon\,\eta_{2}y^{3}}{\sqrt{\eta_{2}y^{4}-\eta_{3}}}\,.
	\tag{iii}
	\end{equation}

	Next, we substitute \eqref{ii} and \eqref{iii} into \eqref{i} to deduce
	\begin{align*}
	\qquad
	& y\Bigl(\eta_1 y^{2} + \varepsilon\sqrt{\eta_{2}y^{4}-\eta_{3}}\Bigr)
	   \Bigl(2\eta_{1}y + \frac{2\varepsilon\eta_{2}y^{3}}
	   {\sqrt{\eta_{2}y^{4}-\eta_{3}}}\Bigr)  \tag{6} \\
	&= \Bigl(\eta_1 y^{2} + \varepsilon\sqrt{\eta_{2}y^{4}-\eta_{3}}\Bigr)^{2}
	   + \tilde{\beta}_{1} y^{2}\Bigl(\eta_1 y^{2} + \varepsilon\sqrt{\eta_{2}y^{4}-\eta_{3}}\Bigr)
	   - \tilde{\beta}_{2}y^{4} - \delta \\
	&= \eta_{1}^{2}y^{4}
	   + 2\varepsilon\eta_{1}y^{2}\sqrt{\eta_{2}y^{4}-\eta_{3}}
	   + \eta_{2}y^{4} - \eta_{3} \\
	&\quad\; + \tilde{\beta}_{1}\eta_{1}y^{4}
	   + \tilde{\beta}_{1}\varepsilon y^{2}\sqrt{\eta_{2}y^{4}-\eta_{3}}
	   - \tilde{\beta}_{2}y^{4} - \delta \\
	&= \bigl(\eta_{1}^{2}+\eta_{2}+\tilde{\beta}_{1}\eta_{1}-\tilde{\beta}_{2}\bigr) y^{4}
	   - (\eta_{3}+\delta)
	   + \bigl(2\varepsilon\eta_{1}+\tilde{\beta}_{1}\varepsilon\bigr)
		 y^{2}\sqrt{\eta_{2}y^{4}-\eta_{3}}.
	\end{align*}
	Hence, we have
	\[ \begin{aligned}
	&\bigl(\eta_{1}^{2}+\eta_{2}+\tilde{\beta}_{1}\eta_{1}-\tilde{\beta}_{2}\bigr) y^{4}
	   - (\eta_{3}+\delta)
	   + \bigl(2\varepsilon\eta_{1}+\tilde{\beta}_{1}\varepsilon\bigr)
		 y^{2}\sqrt{\eta_{2}y^{4}-\eta_{3}} \\ 
	\qquad
	&= y\Biggl(
		   2\eta_{1}^{2}y^{3}
		   + \frac{2\varepsilon \eta_{2}^{2}y^{5}}{\sqrt{\eta_{2}y^{4}-\eta_{3}}}
		   + 2\varepsilon\eta_{1}y\sqrt{\eta_{2}y^{4}-\eta_{3}}
		   + 2\eta_{2}y^{3}
		 \Biggr).
	\end{aligned}
	\]
	Multiplying both sides by $\sqrt{\eta_{2}y^{4}-\eta_{3}}$ and rearranging the terms yields

	\[
	([(\tilde{\beta_1} - \eta_1) \eta_1 - \eta_2 - \tilde{\beta_2}] y^4 - (\eta_3 + \delta))^2 (\eta_2 y^4 - \eta_3) - y^4 ([2 \eta_1 - \tilde{\beta_1}] \eta_2 y^4 + \tilde{\beta_1} \eta_3)^2 = 0.
	\]

	This expression can be rearranged, giving rise to the equality $P(s)=0$, where $P(x)=a_{12}x^{12}+a_8x^8+a_4x^4+a_0$. In what follows we will show that the polynomial $P(x)$ is nonconstant and has degree 12. This will be accomplished once we show that $a_{12}$ is not zero. A straightforward computation shows that
	\[
	a_{12} =  ([\tilde{\beta_1} - \eta_1] \eta_1 - \eta_2 - \tilde{\beta_2})^2 \eta_2 - (2 \eta_1 - \tilde{\beta_1})^2 \eta_2^2.
	\]
	Noting that both terms possess a common factor of $\eta_2$, we factor it to obtain
	\begin{align*}
	a_{12}
	&=  \eta_2 (X^2 - Y^2 \eta_2),
	\end{align*}
	with
	\begin{align*}
	X &= (\tilde{\beta}_1 - \eta_1)\eta_1 - \eta_2 - \tilde{\beta}_2\ \ \ \ \text{and}\ \ \ \ Y = 2\eta_1 - \tilde{\beta}_1.
	\end{align*}
    
    In what follows we compute the auxiliary coefficients in terms of $n,\ r_{1},\ r_{2},\ \mu_{1}$ and $C_{1}$. Using the definitions of $\beta_{i}'s$, we can easily check that
    \[
		\gamma_1 = 3 (n-2)(n-1)^2 (r_1-1), \quad
		\gamma_2 = 4 (n-2)^2 (n-1), \quad
		\gamma_3 = (n-2)^2, \quad
		\gamma_4 = C_1^2 (n-1)^2 r_2 \mu_1.
		\]
	Now we use $r_2 = n - 1 - r_1$ and the expressions for $\gamma_{i}'s$ and $\beta_{i}'s$, to compute $\eta_{1}$, $\eta_{2}$, $\tilde{\beta}_1$ and $\tilde{\beta}_2$:
	\begin{align*}
	    &\eta_1 = \frac{\gamma_2}{2\gamma_1}=\frac{4\,(n-2)^2(n-1)}{2\cdot 3\,(n-2)(n-1)^2\,(r_1-1)}=\frac{2\,(n-2)}{3\,(n-1)(r_1-1)},\\
        &\eta_2 = \frac{\gamma_2^2-4\gamma_1\gamma_3}{4\gamma_1^2}=\frac{\Bigl(4\,(n-2)^2(n-1)\Bigr)^2-4\cdot 3\,(n-2)(n-1)^2\,(r_1-1)\,(n-2)^2}{4\Bigl(3\,(n-2)(n-1)^2\,(r_1-1)\Bigr)^2}=\frac{(n-2)\Bigl(4n-3r_1-5\Bigr)}{9\,(n-1)^2\,(r_1-1)^2},\\
        &\tilde{\beta}_1 = \frac{\beta_1}{\beta_3} = \frac{(n-1)\Bigl(3(r_1-1)+4r_2\Bigr)}{(n-1)^2(r_1-1)}=\frac{3(r_1-1)+4\bigl(n-1-r_1\bigr)}{(n-1)(r_1-1)}=\frac{4n-r_1-7}{(n-1)(r_1-1)},\\
        &\tilde{\beta}_2 = \frac{n-2}{(n-1)^2(r_1-1)}.
	\end{align*}
	
    Next, we use the expressions above to compute $X$. First, notice that 
	\begin{align*}
		\tilde{\beta}_1 - \eta_1 &= \frac{4n - r_1 - 7}{(n-1)(r_1-1)} - \frac{2(n-2)}{3(n-1)(r_1-1)}= \frac{3(4n - r_1 - 7) - 2(n-2)}{3(n-1)(r_1-1)}= \frac{10n - 3r_1 - 17}{3(n-1)(r_1-1)}.
	\end{align*}
	Consequently, $X$ can be computed as below:
	\begin{align*}
		X &= \left(\frac{10n - 3r_1 - 17}{3(n-1)(r_1-1)}\right) \eta_1 - \eta_2 - \tilde{\beta}_2 \\
		&= \left(\frac{10n - 3r_1 - 17}{3(n-1)(r_1-1)}\right)\left(\frac{2(n-2)}{3(n-1)(r_1-1)}\right) - \frac{(n-2)(4n - 3r_1 - 5)}{9(n-1)^2(r_1-1)^2} - \frac{n-2}{(n-1)^2(r_1-1)} \\
		&= \frac{2(n-2)(10n - 3r_1 - 17)}{9(n-1)^2(r_1-1)^2} - \frac{(n-2)(4n - 3r_1 - 5)}{9(n-1)^2(r_1-1)^2} - \frac{9(n-2)(r_1-1)}{9(n-1)^2(r_1-1)^2} \\
		&= \frac{(n-2) \left[ 2(10n - 3r_1 - 17) - (4n - 3r_1 - 5) - 9(r_1-1) \right]}{9(n-1)^2(r_1-1)^2} \\
		&= \frac{(n-2) \left[ 16n - 12r_1 - 20 \right]}{9(n-1)^2(r_1-1)^2} \\ 
		&= \frac{4(n-2)(4n - 3r_1 - 5)}{9(n-1)^2(r_1-1)^2}.
	\end{align*}
    
	The computation of $Y$ proceeds as follows:
	\begin{align*}
		Y &= 2\eta_1 - \tilde{\beta}_1= 2\left(\frac{2(n-2)}{3(n-1)(r_1-1)}\right) - \frac{4n - r_1 - 7}{(n-1)(r_1-1)} \\
		&= \frac{4(n-2)}{3(n-1)(r_1-1)} - \frac{3(4n - r_1 - 7)}{3(n-1)(r_1-1)} \\
		&= \frac{-8n + 3r_1 + 13}{3(n-1)(r_1-1)}.
	\end{align*}
	Substituting these simplified forms of $X$ and $Y$ back into the expression $a_{12}  = \eta_2 (X^2 - Y^2 \eta_2)$, we get
	\begin{align*}
		a_{12} &= \eta_2 \left(\frac{16(n-2)^2(4n - 3r_1 - 5)^2}{81(n-1)^4(r_1-1)^4} - \frac{(-8n + 3r_1 + 13)^2 (n-2)(4n - 3r_1 - 5)}{81(n-1)^4(r_1-1)^4}\right) \\ 
		&= \eta_2 \frac{(n-2)(4n - 3r_1 - 5)}{81(n-1)^4(r_1-1)^4} \left[ 16(n-2)(4n - 3r_1 - 5) - (-8n + 3r_1 + 13)^2\right] \\
		&= \eta_2 \frac{(n-2)(4n - 3r_1 - 5)}{81(n-1)^4(r_1-1)^4}[-9(r_1-1)^2] \\ 
		&= \left(\frac{(n-2)(4n - 3r_1 - 5)}{9(n-1)^2(r_1-1)^2}\right) \left(\frac{-(n-2)(4n - 3r_1 - 5)}{9(n-1)^4(r_1-1)^2}\right) \\ 
		&= -\frac{(n-2)^2 (4n - 3r_1 - 5)^2}{81(n-1)^6 (r_1-1)^4}.
	\end{align*}
	This expression shows that $a_{12} \neq 0$. Thus, \(P\) is not constant and has degree 12, as claimed. Furthermore, $P(f'(s))=0$, $\forall s$, finishing the proof.
\end{proof}

\begin{proposition}\label{propo_r1>1}
    If $r_{1}>1$, then any function satisfying \eqref{sit_f_1'} and \eqref{sit_f_2'} simultaneously must be constant.
\end{proposition}

\begin{proof}
    Assume by contradiction that there is a nonconstant function $f$ satisfying equations \eqref{sit_f_1'} and \eqref{sit_f_2'}, and consider the nonconstant polynomial $P$ of degree $12$, given in the previous lemma. Since $P(f'(s))=0$, $\forall s$, we conclude that $P$ must have a real root, and that $f'$ must be constant. By \eqref{sit_f_1'}, we have $f'=0$, which contradicts the assumption that $f$ is nonconstant.
\end{proof}

\section{Acknowledgments}
V. Borges was partially supported by Capes Finance Code 001 and he thanks the Mathematics Department of Universidade de Brasília, where this work was conducted. M. A. R. M. Horácio was supported by Capes Finance Code 001. J. P. dos Santos was partially supported by CNPq 313200/2025-4 and FAPDF 00193-00001678/2024-39.


\begin{thebibliography}{99}

\bibitem{barros4} Barros, A.; Ribeiro Jr, E. \textit{Some characterizations for compact almost Ricci solitons}. Proceedings of the American Mathematical Society, v. 140, n. 3, p. 1033-1040, 2012.



\bibitem{barros2} Barros, A.; Batista, R.; Ribeiro Jr, E. \textit{Rigidity of gradient almost Ricci solitons}. Illinois Journal of Mathematics, 56(4), 1267-1279, 2012.

\bibitem{barros1} Barros, A.; Batista, R.; Ribeiro Jr, E. \textit{Compact almost Ricci solitons with constant scalar curvature are gradient}. Monatshefte für Mathematik, 174, 29-39. 2014.

\bibitem{besse} Besse, A. L. \textit{Einstein manifolds}. Springer Science and Business Media, 2007.



\bibitem{brgs1} Borges, V.; Tenenblat, K. \textit{Ricci almost solitons on semi‐Riemannian warped products}. Mathematische Nachrichten 295, no. 1, 22-43, 2022.


\bibitem{brgs3} Borges, V. \textit{On complete gradient Schouten solitons}. Nonlinear Analysis, 221, p.112883, 2022.

\bibitem{brgs4} Borges, V. \textit{Rigidity of Bach-flat gradient Schouten solitons}. Manuscripta Mathematica, 175(1), pp.409-419, 2024.

\bibitem{brgs2025} Borges, V.; Hor\'acio, M. A. R. M.; Dos Santos, J. P. \textit{The Ricci tensor of a gradient Ricci soliton with harmonic Weyl tensor}. arXiv preprint arXiv:2510.11939, 2025.

\bibitem{bishop} Bishop, R. L.; O'Neill, B. \textit{Manifolds of negative curvature}. Transactions of the American Mathematical Society, v. 145, p. 1-49, 1969.

\bibitem{brasil} Brasil, A.; Costa, E.; Ribeiro Jr, E. \textit{Hitchin–Thorpe inequality and Kaehler metrics for compact almost Ricci soliton}. Annali di Matematica Pura ed Applicata (1923-) 193. 1851-1860, 2014.



\bibitem{brovaz} Brozos-Vázquez, M.; García-Río, E.; Vázquez-Lorenzo, R. \textit{Warped product metrics and locally conformally flat structures.} Matemática Contemporânea. 28(5):91-110, 2005.


\bibitem{cafega} Calviño-Louzao, E.; Fernández-López, M.; García-Río, E.; Vázquez-Lorenzo, R. \textit{Homogeneous Ricci almost solitons}. Israel Journal of Mathematics, 220(2), 531-546, 2017.


\bibitem{cachen} Cao, H. D.; Chen, Q. \textit{On locally conformally flat gradient steady Ricci solitons}. Transactions of the American Mathematical Society, 364(5), pp.2377-2391, 2012.

\bibitem{caoqian} Cao, H. D.; Chen, Q. \textit{On Bach-flat gradient shrinking Ricci solitons}. Duke Mathematical Journal 162, no. 6, 1149-1169, 2013.

\bibitem{cccmm} Cao, H. D.; Catino, G.; Chen, Q.; Mantegazza, C.; Mazzieri, L. \textit{Bach-flat gradient steady Ricci solitons}. Calculus of Variations and Partial Differential Equations, 49, pp.125-138, 2014.

\bibitem{caozhou} Cao, H. D.; Zhou, D. \textit{On complete gradient shrinking Ricci solitons}. Journal of Differential Geometry, v. 85, n. 2, p. 175-186, 2010.

\bibitem{catino0} Catino, G.  \textit{Generalized quasi-Einstein manifolds with harmonic Weyl tensor}. Mathematische Zeitschrift, 271(3), pp.751-756, 2012.

\bibitem{catino2} Catino, G.; Cremaschi, L.; Djadli, Z.; Mantegazza, C.; Mazzieri, L. \textit{The Ricci–Bourguignon flow}. Pacific Journal of Mathematics, v. 287, n. 2, p. 337-370, 2017.

\bibitem{catino} Catino, G.; L. Mazzieri, \textit{Gradient Einstein solitons}. Nonlinear Analysis 132, 66-94, 2016.

\bibitem{choui} Chouikha, A. R. \textit{Existence of Metrics with Harmonic Curvature and Non Parallel Ricci Tensor}. Balkan Journal of Geometry and Its Applications 8, no. 2, 21-30, 2003.

\bibitem{danwan} Dancer, A. S.; Wang M. Y.; \textit{Some new examples of non-Kähler ricci solitons.} Mathematical Research Letters 16, no. 2, 2009.

\bibitem{derd} Derdziński, A. \textit{Classification of certain compact Riemannian manifolds with harmonic curvature and non-parallel Ricci tensor}. Mathematische Zeitschrift, 172(3), 273-280, 1980.

\bibitem{derd2} Derdziński, A. \textit{On compact Riemannian manifolds with harmonic curvature}. Math. Ann, 259(2), pp.145-152, 1982.

\bibitem{DobUn} Dobarro, F.; Ünal, B. \textit{Curvature of multiply warped products}. Journal of Geometry and Physics, 55(1), pp.75-106, 2005.

\bibitem{fega} Fernández-López, M; García-Río, E. \textit{Rigidity of shrinking Ricci solitons}. Mathematische Zeitschrift, 269, pp.461-466, 2011.


\bibitem{hamilton2} Hamilton, R. S. \textit{The formation of singularities in the Ricci flow}. Surveys in Diff. Geom., 2, pp.7-136, 1995.

\bibitem{kerbrat} Kerbrat Y. \textit{Transformations conformes des variétés pseudo-riemanniennes}. Journal of Differential Geometry. Jan; 11(4):547-71, 1976.

\bibitem{kim1} Kim, J. \textit{On a classification of 4-d gradient Ricci solitons with harmonic Weyl curvature}. The Journal of Geometric Analysis, 27(2), 986-1012, 2017.

\bibitem{kim2} Kim, J. \textit{Four‐dimensional gradient almost Ricci solitons with harmonic Weyl curvature}. Mathematische Nachrichten, 294(4), 774-793, 2021.

\bibitem{kim3} Kim, J. \textit{Classification of Gradient Ricci Solitons with Harmonic Weyl Curvature}. The Journal of Geometric Analysis, 35(5), p.139, 2025.

\bibitem{feng} Li, F. \textit{Rigidity of complete gradient steady Ricci solitons with harmonic Weyl curvature}. Pacific Journal of Mathematics, 335(2), pp.323-353, 2025.

\bibitem{ma} Maschler, G. \textit{Almost soliton duality}. Advances in Geometry, v. 15, n. 2, p. 159-166, 2015.

\bibitem{muse} Munteanu, O.; Sesum, N. \textit{On gradient Ricci solitons}. Journal of Geometric Analysis, 23, pp.539-561, 2013.

\bibitem{oneill} O'Neill, B. \textit{Semi-Riemannian Geometry With Applications to Relativity} 103. Academic Press, 1983.

\bibitem{petersen} Petersen, P.; Wylie, W. \textit{Rigidity of gradient Ricci solitons}. Pacific Journal of Mathematics, v. 241, n. 2, p. 329-345, 2009.

\bibitem{pigola} Pigola, S.; Rigoli, M.; Rimoldi, M.; Setti, A. G. \textit{Ricci almost solitons}. Ann. Sc. Norm. Super. Pisa Cl. Sci. v. 10 n. 4, 757-799, 2011.

\bibitem{shin} Shin, J. \textit{On the classification of 4-dimensional $(m,\rho)$-quasi-Einstein manifolds with harmonic Weyl curvature}. Annals of Global Analysis and Geometry, 51(4), 379-399, 2017.

\end{thebibliography}
\end{document}